\documentclass[review]{elsarticle}
\usepackage{lineno,hyperref}
\usepackage{amsmath, amssymb}
\usepackage{amsmath}
\usepackage{xcolor}
\usepackage{mathtools}    
\usepackage{cases}        
\usepackage{amssymb}
\usepackage{mathdots}
\usepackage{geometry}
\usepackage{amsthm}
\usepackage{subfig}
\usepackage{graphicx}
\usepackage{booktabs}
\usepackage{caption}
\usepackage{placeins}
\newcommand{\R}{\mathbb{R}}
\newcommand{\C}{\mathbb{C}}

\newcommand{\rk}{\operatorname{rank}}

\newcommand{\glt}{\mathrm{GLT}}
\DeclarePairedDelimiter{\norm}{\lVert}{\rVert}
\DeclarePairedDelimiter{\abs}{\lvert}{\rvert}

\theoremstyle{plain}
\newtheorem{gltax}{GLT Axiom}
\newtheorem{theorem}{Theorem}[section]

\newtheorem{corollary}[theorem]{Corollary}

\newtheorem{Remark}[theorem]{Remark}
\newtheorem{Definition}[theorem]{Definition}
\modulolinenumbers[5]

\journal{Journal of \LaTeX\ Templates}

\begin{document}

\begin{frontmatter}

\title{Efficient Krylov solvers for inverse source problem in 2D space-time fractional diffusion equation}
\author[label1]{Asim Ilyas$^*$}\ead{ailyas1@uninsubria.it}
\author[label1,label4]{Stefano Serra-Capizzano}\ead{s.serracapizzano@uninsubria.it}
\address[label1]{Department of Science and High Technology, University of Insubria, Como, Italy}
\address[label4]{Department of Information Technology, Uppsala University, Uppsala, Sweden}
\begin{abstract}
In this work, we consider a two-dimensional time-space fractional diffusion equation with a variable coefficient and investigate the inverse source problem of reconstructing the source term \( f(x,y) \), after regularizing the problem using the quasi-boundary value method to mitigate ill-posedness. A finite difference discretization results in a large-scale linear system with a multilevel Toeplitz-like block structure. We perform a spectral analysis of the associated matrix sequences, employing tools from Generalized Locally Toeplitz (GLT) theory, and construct efficient preconditioners based on the GLT analysis. The proposed preconditioners preserve the multilevel structure of the discretization matrices and leads to a general eigenvalue clustering around one for the preconditioned sequence. Numerical experiments validate the theoretical findings and demonstrate that the proposed approach significantly accelerates the convergence of the GMRES method in reconstructing the source term in the two-dimensional space-time fractional diffusion equation.
\end{abstract}
\begin{keyword}
\texttt{Fractional derivatives, Inverse problem, Quasi-boundary value method, GMRES, Generalized Locally Toeplitz, Preconditioning}
\end{keyword}

\end{frontmatter}

\section{Introduction}
In this article, we are concerned with the design of a fast solver for the numerical approximation of the inverse source problem of identifying the source term $f(x,y)$ from the space-time fractional diffusion equation  	
\begin{equation}\label{eq:problem}
	\begin{dcases} {}^{T}D_{t}^{\xi,\rho}u(x,y,t)+ a(x,y)\left({}^{C}D_{x}^{\eta_{1}}+{}^{C}D_{y}^{\eta_{2}}\right)u(x,y,t)=q(t)f(x,y), & (x,y,t) \in \Omega_{T},\\	u(0,y,t)=0=u(1,y,t), & y \in (0,1), \quad t\in(0,T),\\
		u(x,0,t)=0=u(x,1,t), & x \in (0,1), \quad t\in(0,T),\\
		u(x,y,0) = \varphi(x,y), & x \in (0,1), \quad y \in (0,1), \\
		u(x,y,T)=\phi(x,y),& x\in(0,1), \quad y\in(0,1).\end{dcases}\end{equation}
In the previous set of equations, the term ${}^{T}D_{t}^{\xi,\rho}$ represents the tempered fractional derivative of order $\xi\in(0,1)$ and $\rho>0$ is the truncation parameter, $a(x,y)$ is a continuous variable coefficient such that $a(x,y)\neq 0$, ${}^{C}D_{x}^{\eta_{1}}$\; and $\;{}^{C}D_{y}^{\eta_{2}}$ denote the Caputo fractional derivatives of orders $\eta_{1}, \eta_{2}\in(0,1)$, $q(t)$ is a given function, $\varphi(x,y)$ is the initial condition, and $\phi(x,y)$ is the final time datum and $\Omega_{T}:=(0,1)\times(0,1)\times (0,T)$, respectively.
The tempered fractional derivative in time variable of order $\xi\in(0,1)$ is defined as follows
$${}^{T}D_{t}^{\xi,\rho}u(x,y,t):=\frac{1}{\Gamma(1-{\xi})}\int_{0}^{t}e^{-\rho(t-\tau)}(t-\tau)^{-{\xi}}\frac{\partial}{\partial \tau}u(x,y,\tau)d\tau.$$
The Caputo fractional derivative in the space variable of orders $\eta\in(0,1)$ is defined as follows
\[{}^{C}D^{\eta}_{x}u(x,y,t)=\frac{1}{\Gamma(1-\eta)}\int\limits^{x}_{0}{(x-\zeta)^{-\eta}}\frac{\partial}{\partial \zeta}{u(x,y,\zeta)}\;d\zeta,\qquad x >0,\]

The motivation for this work stems from the increasing importance of space-time fractional diffusion equations (FDEs), which offer a versatile framework for modeling complex, multiscale phenomena in a rich variety of scientific disciplines. By incorporating fractional derivatives, FDEs can capture essential features such as memory effects, non-local interactions, and anomalous diffusion behaviors. These characteristics make FDEs ideal for modeling processes that involve complex transport phenomena, including anomalous diffusion, subdiffusion, and superdiffusion in heterogeneous media. Notable applications include fields such as viscoelasticity \cite{Zhou2022}, anomalous diffusion \cite{Kilbas2006,Metzler2014}, biological systems \cite{Ghafoor,Owolabi2024}, chemistry \cite{Metzler2025}, finance \cite{Muhammad2024}, and chaotic systems \cite{Gomez-Aguilar}. In particular, the Caputo fractional derivative is widely used for modeling spatial heterogeneities, as it incorporates physically meaningful initial conditions and avoids singularities at the boundaries. Meanwhile, tempered fractional derivatives are effective in describing transient anomalous diffusion, where the diffusion initially follows a fractional power law and later transitions to normal diffusion \cite{Sandev2017}. Equation \eqref{eq:problem} incorporates both tempered memory effects in time and fractional diffusion in two spatial dimensions, making it well-suited for modeling complex transport phenomena such as anomalous dispersion in heterogeneous media.

Inverse problems are ubiquitous in practical applications, where the goal is to identify unknown causes from observed effects. However, these problems are often ill-posed in the sense of Hadamard, meaning that small perturbations in the input data can lead to significant errors in the solution. As a result, various regularization methods have been developed to transform these ill-posed problems into well-posed ones. Specifically, inverse problems in the context of fractional differential equations (FDEs) have been the subject of extensive analytical and computational research \cite{Ilyas2024, Jin2015, Tuan2017, Malik2024, Rooh2024, Suhaib2023}. These inverse source problems arise in fields such as environmental modeling, medical imaging, and materials science, where accurate identification of underlying source terms is essential for predictive modeling and system optimization. In this work, we address the ill-posedness of the inverse source problem \eqref{eq:problem} in the context of a two-dimensional space-time FDE by applying the quasi-boundary value method, a nonlocal regularization technique. This approach has been successfully utilized to handle similar ill-posed problems in both parabolic and fractional diffusion equations \cite{Denche2005, Hao2008}. We also recognize that the measurement of the final time datum is inevitably affected by noise. Let us denote the noisy datum by $\phi_\varepsilon(x,y)$, where $\varepsilon>0$ represents the noise level. Assuming the noise is sufficiently small, we impose the condition:
\begin{equation*}
	\norm{\phi - \phi_\varepsilon}_{L^2} \leq \varepsilon,
\end{equation*}
where $\norm{\cdot}_{L^2}$ is the $L^2$ norm on $(0,1)\times(0,1)$. When we apply the quasi-boundary value method, we obtain a class of parameter dependent well-posed regularized problems
\begin{equation}\label{eq:quasi-problem}
	\begin{dcases} {}^{T}D_{t}^{\xi,\rho}u(x,y,t)+ a(x,y)\left({}^{C}D_{x}^{\eta_{1}}+{}^{C}D_{y}^{\eta_{2}}\right)u(x,y,t)=q(t)f_{\varepsilon,\lambda}(x,y), & (x,y,t) \in \Omega_{T},\\	u(0,y,t)=0=u(1,y,t), & y \in (0,1), \quad t\in(0,T),\\
		u(x,0,t)=0=u(x,1,t), & x \in (0,1), \quad t\in(0,T),\\
		u(x,y,0) = \varphi(x,y), & x \in (0,1), \quad y \in (0,1), \\
		u(x,y,T)=\phi_{\varepsilon}(x,y) - \lambda\, a(x,y) \left({}^{C}D_{x}^{\eta_{1}}+{}^{C}D_{y}^{\eta_{2}}\right) f_{\varepsilon,\lambda}(x,y),& x\in(0,1), \quad y\in(0,1),\end{dcases}
\end{equation}
which approximate the original FDE in \eqref{eq:problem}.

Here, $\lambda>0$ is a regularization parameter, $f_{\varepsilon,\lambda}$ represents the regularized approximation of the source term $f(x,y)$ and the solution $u(x,y,t)\equiv u_{\varepsilon,\lambda}(x,y,t)$ provides an approximation to the original solution of problem \eqref{eq:problem}. The regularization term $\lambda\, a(x,y) \left({}^{C}D_{x}^{\eta_{1}}+{}^{C}D_{y}^{\eta_{2}}\right)$ ensures that the reconstructed source term is stable and well-behaved, even in the presence of noise.

In this paper, we focus on the development of a fast algorithm to reconstruct the source term \( f(x, y) \) in the context of a two-dimensional space-time FDE, starting from the regularized formulation \eqref{eq:quasi-problem}, and generalizing the work in \cite{Ilyas2026} devoted to unidimensional space version of \eqref{eq:problem}-\eqref{eq:quasi-problem}. Numerous efficient numerical approaches have been devised for solving FDEs, including finite difference schemes \cite{Chen2014, Hao2021, Thomas2013}, finite element methods \cite{Bu2014, Dhatt2012, Li2018}, finite volume techniques \cite{Eymard2000, Simmons2017}, and spectral methods \cite{Gottlieb1977}. Despite their effectiveness, achieving high accuracy typically necessitates very fine discretizations in both space and time, which can lead to large, intricate systems of equations that are computationally expensive. This is further complicated by the nonlocal behavior inherent in fractional differential operators, resulting in dense matrices. Nevertheless, discretizing these equations on a uniform grid often results in coefficient matrices with a Toeplitz-like structure, which allows for the use of faster and more efficient solvers, such as preconditioned iterative methods, to reduce computational time \cite{Donatelli2016, Lei2013, Pang2021}.

The extension of these methods to a two-dimensional problem introduces additional challenges due to the increased dimensionality of the system. However, discretizing the two-dimensional problem using finite differences leads to a multilevel block Toeplitz structure, enabling the use of advanced solvers. Beyond preconditioned iterative methods, direct solvers that leverage matrix structures have proven effective in improving computational efficiency. For example, divide-and-conquer techniques have been successfully applied to time-dependent FDEs with spatially varying diffusion coefficients, significantly accelerating the solution process \cite{Jin2023}. These numerical advancements, combined with matrix structure exploitation, enable the efficient solution of large-scale two-dimensional FDEs, ensuring both speed and accuracy in real-world applications.

To solve this problem numerically, our proposal is to apply a finite difference scheme, leading to a two-by-two block linear system, and then to design symbol-based approximations \cite{Ilyas2026,Rosita-SIMAX} to the resulting coefficient matrices. More precisely, our primary focus is on the design of an effective preconditioner for the Generalized Minimal RESidual (GMRES) method \cite{Saad1986GMRES}. For the preconditioner to be efficient, it must approximate the original matrix in a spectral sense. We begin by conducting an asymptotic spectral analysis of the coefficient matrix sequences. Due to the presence of the variable coefficient \( a(x, y) \), which disrupts the Toeplitz structure, we rely on GLT matrix theory to address this challenge. GLT theory plays a crucial role in managing the complexity of the variable coefficient in two-dimensional problems, as discussed in \cite{garoni2017, garoni2018, barbarino2020uni, barbarino2020multi}. For a comprehensive guide to applying GLT theory in numerical problems, refer to \cite{tom, glt-tutorial}. The spectral analysis provides the foundation for developing an effective preconditioning strategy, ensuring the preconditioned matrix sequence is an accurate approximation of the original coefficient matrix.

The structure of the paper is as follows. Section \ref{sec:prelim} presents the theoretical framework and an overview of the key tools used in the analysis. Section \ref{sec:discre} outlines the discretization of the regularized problem using a finite difference scheme for both space and time, and discusses the structural properties of the resulting linear system. In Section \ref{sec:glt-analysis}, we analyze the coefficient matrix sequence using GLT theory. Section \ref{sec:precond} extends this analysis to the preconditioned matrix sequences. In Section \ref{sec:num}, we conduct numerical experiments to evaluate the performance of the proposed preconditioner and we report a detailed analysis of the computational costs. Finally, Section \ref{sec:end} provides concluding remarks.

\section{Preliminaries} \label{sec:prelim}

We briefly introduce the fundamental tools employed in this work, which are discussed in detail in \cite{garoni2017, garoni2018, barbarino2020uni, barbarino2020multi}. Here, we present the definitions and theorems directly in the multilevel, multi-dimensional setting, which is required for the analysis of the current paper.


Let us first establish some useful notation.
\begin{itemize}
	\item Given a square matrix $A_n$, we denote with $\lambda_j(A_n)$ and $\sigma_j(A_n)$ the $j$-th eigenvalue and singular value of $A_n$, respectively, as $j=1,\ldots,n$.
	\item The spectral norm for a square matrix $A_n$ is denoted as $\norm{A_n}$.
	\item Whenever we use terminology from measure theory, for instance “measurable set”, “measurable function”, “a.e.”, we always refer to the Lebesgue measure in $\R^t$, denoted with $\mu_t$.
\end{itemize}

\noindent\textbf{Multi-index notation}. To effectively deal with multilevel structures it is necessary to use multi-indices, which are vectors of the form $\textbf{i} = (i_1, \cdots, i_d) \in \mathbb{Z}^d$. The related notation is listed below.
\begin{itemize}
	\item $\mathbf{0}, \mathbf{1}, \mathbf{2}, \dots$ are vectors of all zeroes, ones, twos, etc.
	\item $\textbf{h} \leq \textbf{k}$ means that $h_r \leq k_r$ for all $r = 1, \cdots, d$. In general, relations between multi-indices are evaluated componentwise.
	\item Operations between multi-indices, such as addition, subtraction, multiplication, and division, are also performed componentwise.
	\item The multi-index interval $[\textbf{h}, \textbf{k}]$ is the set $\{\textbf{j} \in \mathbb{Z}^d : \textbf{h} \leq \textbf{j} \leq \textbf{k}\}$. We always assume that the elements in an interval $[\textbf{h}, \textbf{k}]$ are ordered in the standard lexicographic manner
	\[
	\Biggr[ \cdots
	\biggr[\Bigr[\bigl(j_1, \cdots, j_d\bigl)\Bigr]
	_{j_d = h_d, \cdots, k_d}\biggr]
	_{j_{d-1} = h_{d-1}, \cdots, k_{d-1}}
	\cdots \Biggr]_{j_1 = h_1, \cdots, k_1}.
	\]
	\item $\textbf{j} = \textbf{h}, \cdots, \textbf{k}$ means that $\textbf{j}$ varies from $\textbf{h}$ to $\textbf{k}$, always following the lexicographic ordering.
	\item $\textbf{m} \to \infty$ means that $\min(\textbf{m}) = \min_{j=1, \cdots, d} m_j \to \infty$.
	\item The product of all the components of $\textbf{m}$ is denoted as $\nu(\textbf{m}) := \prod_{j=1}^{d} m_j$.
\end{itemize}

A \textbf{multilevel matrix sequence} is a matrix sequence $\{A_{\textbf{n}} \}_n$ such that $n$ varies in some infinite subset of $\mathbb{N}$, $\textbf{n} = \textbf{n}(n)$ is a multi-index in $\mathbb{N}^d$ depending on $n$, and $\textbf{n} \to \infty$ when $n \to \infty$. This is typical of many approximations of differential operators in $d$ dimensions.

\subsection{Spectral tools}

Throughout this paper, we work with matrix sequences, which refer to any sequence of the form \( \{A_n\}_n \), where each \( A_n \) is a square matrix of size \( d_n \), with \( d_n \to \infty \) monotonically as \( n \to \infty \). For a given matrix sequence, it is often possible to associate a spectral or singular value distribution, as defined below.
\begin{Definition} \label{def:distrib}
	Let $\{A_n\}_n$ be a matrix sequence, where $A_n$ has size $d_n\times d_n$, and let $\psi:D\subset\R^t \to\C^{r\times r}$ be a measurable function defined on a set $D$ with $0 < \mu_t(D) < \infty$. Denote with $C_c(\mathbb{K})$ the set of continuous complex-valued functions with bounded support on $\mathbb{K}\in\{\R,\C\}$.
	\begin{itemize}
		\item $\{A_n\}_n$ has an \emph{(asymptotic) singular value distribution} described by $\psi$ if for any $F\in C_c(\R)$
		    \begin{equation*}
			\lim_{n \to \infty} \frac{1}{d_n} \sum_{i=1}^{d_n} F(\sigma_i(A_n)) = \frac{1}{\mu_t(D)} \displaystyle \int_D \displaystyle\frac{\sum_{i=1}^{r} F(\sigma_i(\psi(\mathbf{x})))}{r} d\mathbf{x},
		\end{equation*}
		where $\sigma_i(A_n)$, $i=1,\ldots,d_n$, are the singular values of $A_n$. We use the notation $\{A_n\}_n \sim_\sigma \psi$.
		\item $\{A_n\}_n$ has an \emph{(asymptotic) spectral} or \emph{eigenvalue distribution} described by $\psi$ if for any $F\in C_c(\C)$
		 \begin{equation*}
			\lim_{n \to \infty} \frac{1}{d_n} \sum_{i=1}^{d_n} F(\lambda_i(A_n)) = \frac{1}{\mu_t(D)} \int_D \displaystyle\frac{\sum_{i=1}^{r} F(\lambda_i(\psi(\mathbf{x})))}{r} d\mathbf{x},
		\end{equation*}
		where $\lambda_i(A_n)$, $i=1,\ldots,d_n$, are {the eigenvalues of $A_n$}. We use the notation $\{A_n\}_n \sim_\lambda \psi$.
	\end{itemize}
\end{Definition}

The informal meaning behind Definition \ref{def:distrib} is the following: if $\psi$ is continuous and $n$ is large, then the eigenvalues or singular values of $A_n$ (suitably ordered and possibly except for $o(d_n)$ outliers) are approximated by a uniform sampling of $\psi$ over its domain.

The notion of clustering can be seen a special case of distribution. In what follows, we denote the $\varepsilon$-expansion of a set $S\subseteq\C$ as
\begin{equation*}
	B(S,\varepsilon) := \bigcup_{z\in S} B(z,\varepsilon),
\end{equation*}
where $B(z,\varepsilon) := \{w\in\C : \abs{w-z} < \varepsilon\}$ is the complex disk with center $z$ and radius $\varepsilon > 0$.

\begin{Definition} \label{clustering}
	Let $\{A_n\}_n$ be a matrix sequence, with $A_n$ of size $d_n\times d_n$, and let $S\subseteq\C$ be nonempty and closed. The sequence $\{A_n\}_n$ is \emph{strongly clustered} at $S$ in the sense of the eigenvalues if $\forall\varepsilon >0$
	\begin{equation*}
		\# \{ j\in\{1,\ldots,d_n\} : \lambda_j(A_n) \notin B(S, \varepsilon) \} = O(1),
		\qquad n\to\infty.
	\end{equation*}
	In other words, the number of eigenvalues of $A_n$ outside $B(S,\varepsilon)$ is bounded by a constant independent of $n$. Moreover, $\{A_n\}_n$ is \textit{weakly clustered} at $S$ if $\forall\varepsilon >0$
	\begin{equation*}
		\# \{ j\in\{1,\ldots,d_n\} : \lambda_j(A_n) \notin B(S, \varepsilon) \} = o(d_n),
		\qquad n\to\infty,
	\end{equation*}
	meaning that the number of eigenvalues of $A_n$ outside $B(S,\varepsilon)$ is negligible with respect to the size of the matrix. A corresponding definition can be given for the singular values, with $S\subseteq\R^+$.
\end{Definition}

\begin{Remark} \label{remark:distrib-cluster}
	Let us clarify the relationship between the concepts of distribution and clustering. We recall that the essential range of a measurable function $g:D\subseteq\R^t\to\C$ is the set
	\begin{equation*}
		\{ z\in\C : \mu_t \big(\{ g \in B(z, \varepsilon) \} \big)> 0 \quad\forall\varepsilon > 0 \}.
	\end{equation*}
	As reported in \cite[Theorem 4.2]{Golinskii2007}, it holds
	\begin{equation*}
		\{A_n\}_n \sim_\lambda \psi
		\qquad\Longrightarrow\qquad
		\{A_n\}_n \text{ is weakly clustered at the essential range of } \psi.
	\end{equation*}
	Furthermore, if the essential range of $\psi$ consists of a single number $s\in\C$, then
	\begin{equation*}
		\{A_n\}_n \sim_\lambda \psi
		\qquad\iff\qquad
		\{A_n\}_n \text{ is weakly clustered at } s \text{ in the sense of the eigenvalues}.
	\end{equation*}
	The latter case is typically of interest in the context of preconditioning. Corresponding statements can be given for the singular values, with the obvious suitable changes.
\end{Remark}

\subsection{The multilevel block GLT \texorpdfstring{$\ast$}{}-algebra}

A multilevel block GLT sequence is a matrix sequence belonging to the $\ast$-algebra generated by the three specific classes of matrix sequences: zero-distributed, multilevel block Toeplitz and multilevel block diagonal sampling matrix sequences. We define them in the following paragraphs and they can be seen as the basic building blocks of the GLT $\ast$-algebra.

Let \( r \geq 1 \) be a fixed integer. A multilevel \( r \)-block GLT sequence, or simply a GLT sequence when \( r \) is not specified, is a special multilevel \( r \)-block matrix sequence that is associated with a measurable function \( \kappa : [0,1]^d \times [-\pi, \pi]^d \to \mathbb{C}^{r \times r} \), where \( d \geq 1 \), and is called the GLT \emph{symbol}. When $d=1$, we are in the context of unilevel GLT sequences, while $d>1$ relates to the multilevel case. The symbol is essentially unique, in the sense that if $\kappa, \varsigma$ are two symbols of the same  GLT sequence, then $\kappa = \varsigma$ a.e. We write \( \{A_n\}_n \sim_{\mathrm{GLT}} \kappa \) to denote that $\{A_n\}_n$ is a GLT
sequence with symbol $\kappa$.
\vspace{0.3cm}

\noindent\textbf{Zero-distributed sequences:} Zero-distributed sequences are defined as matrix sequences $\{A_n\}_n$ such that $\{A_n\}_n \sim_\sigma 0$. Note that, for any $r \geq 1$, $\{A_n\}_n \sim_\sigma 0$ is equivalent to
$\{A_n\}_n \sim_\sigma O_r$, where $O_r$ is the $r \times r$ zero matrix. The following theorem, taken from \cite{garoni2017,  Serra2001}, provides a useful characterization for detecting this type of sequences. We use
the natural convention $1/\infty = 0$.
\begin{theorem}
	Let $\{A_n\}_n$ be a matrix sequence, with $A_n$ of size $d_n$. Then
	\begin{itemize}\label{lemma:zero-distr}
		\item $\{A_n\}_n \sim_\sigma 0$ if and only if $A_n = R_n + N_n$ with ${\operatorname{rank}(R_n)}/{d_n} \to 0$ and $ ||N_n|| \to 0$ as $n \to \infty$;
		\item $\{A_n\}_n \sim_\sigma 0$ if there exists $p \in [1,\infty]$ such that $
		{\|A_n\|_p}/{(d_n)^{1/p}} \to 0$ as $n \to \infty.$
	\end{itemize}
\end{theorem}
\vspace{0.2cm}

\noindent\textbf{ Multilevel block Toeplitz matrices:} Given \(\textbf{n} \in \mathbb{N}^d \), a matrix of the form
\begin{equation*}
	[A_{\textbf{i}-\textbf{j}}]_{\textbf{i,j=1}}^{\textbf{n}} \in \mathbb{C}^{\nu(\textbf{n})r \times \nu(\textbf{n})r},
\end{equation*}
with blocks \( A_\textbf{k} \in \mathbb{C}^{r \times r} \), \( \textbf{k} \in \{-(\textbf{n}-1), \dots, \textbf{n}-1\} \), is called a multilevel block Toeplitz matrix, or, more precisely, a \( d \)-level \( r \)-block Toeplitz matrix.\\

Given a matrix-valued function \( f : [-\pi,\pi]^d \to \mathbb{C}^{r \times r} \) belonging to \( L^1([- \pi, \pi]^d) \), the \( \mathbf{n} \)-th Toeplitz matrix associated with \( f \) is defined as
\begin{equation*}
	T_\mathbf{n}(f):=[\hat{f}_{\mathbf{i-j}}]_{\mathbf{i,j=1}}^{\mathbf{n}} \in \mathbb{C}^{\nu(\mathbf{n})r \times \nu(\mathbf{n})r},
\end{equation*}
where
\begin{equation*}
	\hat{f}_\mathbf{k} = \frac{1}{(2\pi)^d} \int_{[-\pi,\pi]^d} f(\mathbf{\theta}) e^{-\hat{\iota} (\mathbf{k}, \mathbf{\theta)}} d\mathbf{\theta} \in \mathbb{C}^{r \times r}, \quad \mathbf{k} \in \mathbb{Z}^d,
\end{equation*}
are the Fourier coefficients of \( f \), in which \( \hat{\iota} \) denotes the imaginary unit, the integrals are computed componentwise, and \( (\mathbf{k}, \mathbf{\theta}) = k_1\theta_1 + \cdots + k_d\theta_d \). Equivalently, \( T_\mathbf{n}(f) \) can be expressed as
\begin{equation*}
	T_\mathbf{n}(f) =
	\sum_{|j_1|<n_1} \cdots \sum_{|j_d|<n_d} [J_{n_1}^{(j_1)}
	\otimes \cdots \otimes J^{(j_d)}_{n_d}] \otimes \hat{f}(j_1, \cdots, j_d),
\end{equation*}
where \( \otimes \) denotes the Kronecker tensor product between matrices, and \( J^{(l)}_{m} \) is the matrix of order \( m \) whose \( (i,j) \) entry equals 1 if \( i - j = l \) and zero otherwise.\\

\( \{T_{\mathbf{n}}(f)\}_{\mathbf{n}\in\mathbb{N}^d} \) is the family of (multilevel block) Toeplitz matrices associated with \( f \), which is called the generating function.
\par
\vspace{0.3cm}

\noindent\textbf{Block diagonal sampling matrices:} Given $d \geq 1$, $\mathbf{n} \in \mathbb{N}^d$ and a function $a : [0,1]^d \to \mathbb{C}^{r \times r}$, we define the multilevel block diagonal sampling matrix $D_{\mathbf{n}}(a)$ as the block diagonal matrix
\begin{equation}\label{eq:dna}
	D_{\mathbf{n}}(a) = \underset{\mathbf{i=1,\dots,n}}{\operatorname{diag}} a \left( \mathbf{\frac{i}{n}} \right) \in \mathbb{C}^{\nu(\mathbf{n})r \times \nu(\mathbf{n})r}.
\end{equation}
\vspace{0.4cm}

The GLT class satisfies several algebraic and topological properties that are treated in great detail and generality in \cite{barbarino2020uni,barbarino2020multi,garoni2017,garoni2018}. For the purposes of this work, it is sufficient to present them through their operative properties, listed below in the multilevel setting i.e. for $d\geq1$.

\begin{gltax} \label{glt:distrib}
	If $\{A_n\}_n \sim_\glt \kappa$, then $\{A_n\}_n \sim_\sigma \kappa$ in the sense of Definition \ref{def:distrib}, with $t=2d$ and $D=[0,1]^{d}\times[-\pi,\pi]^{d}$. If moreover each $A_n$ is Hermitian, then $\{A_n\}_n \sim_\lambda \kappa$, again in the sense of Definition \ref{def:distrib}, with $t=2d$.
\end{gltax}

\begin{gltax} \label{glt:symbols}
	It holds
	\begin{itemize}
		\item $\{T_\mathbf{n}(f)\}_\mathbf{n} \sim_\glt \kappa(\mathbf{x},\mathbf{\theta}) = f(\theta)$ for any $f: [-\pi,\pi]^d \to \mathbb{C}^{r \times r}$ with
		$f \in L^1([-\pi,\pi]^d)$;
		\item $\{D_\mathbf{n}(a)\}_\mathbf{n} \sim_\glt \kappa(\mathbf{x},\mathbf{\theta}) = a(x)$ for any Riemann-integrable $a:[0,1]^d \to \mathbb{C}^{r \times r}$
		;
		\item $\{Z_n\}_n \sim_\glt \kappa(\mathbf{x},\mathbf{\theta}) = O_r$ if and only if $\{Z_n\}_n \sim_\sigma 0$, i.e., $\{Z_n\}_n$ is zero-distributed.
	\end{itemize}
\end{gltax}

\begin{gltax} \label{glt:algebra}
	If $\{A_n\}_n \sim_\glt \kappa$ and $\{B_n\}_n \sim_\glt \varsigma$, then
	\begin{itemize}
		\item $\{A_n^*\}_n \sim_\glt \bar{\kappa}$, where $A_n^*$ denotes the conjugate transpose of $A_n$ and $\bar{\kappa}$ the complex conjugate of $\kappa$;
		\item $\{\alpha A_n + \beta B_n\}_n \sim_\glt \alpha\kappa +\beta\varsigma$ for any $\alpha,\beta\in\C$;
		\item $\{A_n B_n\}_n \sim_\glt \kappa\varsigma$;
		\item $\{A_n^\dagger\}_n \sim_\glt\kappa^{-1}$ for any $\kappa$ invertible a.e., where $A_n^\dagger$ denotes the Moore-Penrose pseudoinverse of $A_n$.
	\end{itemize}
\end{gltax}

\begin{gltax} \label{glt:kron}
	Let \( \{A_n\}_n \sim_\glt \kappa \) and \( \{B_m\}_m \sim_\glt \xi \) with
	\begin{align*}
		&\kappa(x_1, \theta_1) : [0,1]^d \times [-\pi, \pi]^d \to \mathbb{C}^{r \times r}, \\
		&\xi(x_2, \theta_2) : [0,1]^d \times [-\pi, \pi]^d \to \mathbb{C}^{r' \times r'}.
	\end{align*}
	
	Then, up to permutations, not present if $r=r'=1$, it holds that \( \{A_n \otimes B_m\}_N \sim_\glt \kappa \otimes \xi \), where \( (\kappa \otimes \xi): [0,1]^{2d} \times [-\pi, \pi]^{2d} \to \mathbb{C}^{rr' \times rr'} \) is given by
	\begin{equation*}
		(\kappa \otimes \xi)(x_1, x_2, \theta_1, \theta_2) := \kappa(x_1, \theta_1) \xi(x_2, \theta_2).
	\end{equation*}
\end{gltax}

We notice that the last axiom is not present in the original works \cite{garoni2017,garoni2018,barbarino2020uni,barbarino2020multi}, but is has been added very recently to the theory in \cite{Rosita-SIMAX} and used already in \cite{Ilyas2026,Rosita-SIMAX}.

\section{Discretization and matrix structures} \label{sec:discre}

In the current section, we present the discretization of the regularized problem \eqref{eq:quasi-problem} using a finite difference scheme \cite{Sun2020}. We then carefully discuss the global structure of the resulting coefficient matrix.


Let $n_{1}$, $n_{2}$ and $m$ be positive integers and denote the corresponding spatial and temporal meshes as
\begin{alignat}{3}
	x_i &= i\Delta x, && \qquad\Delta x = \frac{1}{n_{1}+1}, && \qquad i=0,1,\ldots,n_{1}, \label{eq:1space-mesh} \\
		y_j &= j\Delta y, && \qquad\Delta y = \frac{1}{n_{2}+1}, && \qquad j=0,1,\ldots,n_{2}, \label{eq:2space-mesh} \\
	t_s &= s\Delta t, && \qquad\Delta t=\frac{T}{m}, && \qquad s=0,1,\ldots,m. \label{eq:time-mesh}
\end{alignat}
The approximate values of the functions at the grid points are denoted as follows
\begin{alignat*}{3}
	u_{i,j}^{s} &:= u(x_i,y_{j},t_s), & \qquad
	a_{i,j} &:= a(x_i,y_{j}), & \qquad
	f_{i,j} &:= f_{\varepsilon,\lambda}(x_i,y_{j}),
	\\
	\varphi_{i,j} &:= \varphi(x_i,y_{j}), & \qquad
	\phi_{i,j} &:= \phi_\varepsilon(x_i,y_{j}), & \qquad
	q^{s} &:= q(t_s).
\end{alignat*}
To deal with the tempered fractional derivative, we use the $L1$ formula:
\begin{equation*}
	{}^T\!D_{L1}^{\xi,\rho} u(x_{i},y_{j},t_{s}) = \frac{(\Delta t)^{-\xi}}{\Gamma(2-\xi)} \left[ b_0 u_{i,j}^{s} - \sum_{l=1}^{s-1} \left( b_{s-l-1} - b_{s-l} \right) e^{\rho(t_l-t_s)} u_{i,j}^{l} - b_{s-1} e^{\rho(t_0-t_s)} u_{i,j)}^{0} \right],
\end{equation*}
in which
\begin{equation*}
	b_l = (l+1)^{1-\xi}-l^{1-\xi},
	\qquad l = 0,1,\ldots
\end{equation*}
Denoting
\begin{equation*}
	\alpha_{m} := (\Delta t)^\xi\Gamma(2-\xi)
	= \bigg(\frac{T}{N}\bigg)^\xi\Gamma(2-\xi)
\end{equation*}
and
\begin{equation}
	\begin{split} \label{eq:time-coeff}
		\gamma_0 &:= b_0, \\
		\gamma_l &:= b_l - b_{l-1} = (l+1)^{1-\xi} - 2l^{1-\xi} + (l-1)^{1-\xi}, \qquad l=1,\ldots,m-1,
	\end{split}
\end{equation}
and using $t_s = s\Delta t$, the $L1$ approximation of the tempered fractional derivative is expressed as
\begin{equation} \label{eq:tem-discre}
	{}^T\!D_{L1}^{\xi,\rho} u(x_{i},y_{j},t_s) = \frac{1}{\alpha_m} \left[ \gamma_0 u_{i,j}^{s} + \sum_{l=1}^{s-1} \gamma_{s-l} e^{-(s-l)\rho\Delta t} u_{i,j}^{l} - b_{s-1} e^{-s \rho \Delta t} u_{i,j}^{0} \right].
\end{equation}
The following error estimate holds; {see \cite[page 10]{Zhao2021}}. Assuming that $u(\cdot,\cdot,t)\in C^2(0,T)$, for $t_s\in (0,T)$ we have
\begin{align*}
	\abs{{}^T\!D_t^{\xi,\rho} u(x,y,t)|_{t=t_s}-{}^T\!D_{L1}^{\xi,\rho}u(x,y,t_s)} \leq& \chi_\xi
	\Big[ \rho^2 \max_{t_0\leq t\leq t_s} \abs{u(x,y,t)} + 2\rho \max_{t_0\leq t\leq t_s} \abs{u_t(x,y,t)} \\+& \Delta t^{2-\xi} \max_{t_0\leq t\leq t_s} \abs{u_{tt}(x,y,t)} \Big],
\end{align*}
where $u_t$ and $u_{tt}$ denote the first and second partial derivative of $u$ with respect to $t$, and
\begin{equation*}
	\chi_\xi :=  \frac{1}{2\Gamma(1-\xi)} \left[ \frac{1}{4} + \frac{\xi}{(1-\xi)(2-\xi)} \right].
\end{equation*}

To discretize the Caputo fractional derivative with respect to $x$, we exploit the $L1$ formula:
\begin{equation} \label{eq:Caputo-xdiscre}
	{}^C\!D_{L1}^{\eta_{1}} u(x_i,y_{j},t_{s}) = \frac{1}{\beta_{n_{1}}} \left[ \delta_0 u_{i,j}^{s} + \sum_{k_{1}=1}^{i-1} \delta_{i-k_{1}} u_{k_{1},j}^{s} - d_{i-1} u_{0,j}^{s} \right],
\end{equation}
where
\begin{equation}\label{eq:beta-n1}
	\beta_{n_1} := (\Delta x)^{\eta_{1}} \Gamma(2-\eta_{1}) = \bigg(\frac{1}{n_{1}+1}\bigg)^{\eta_{1}} \Gamma(2-\eta_{1})
\end{equation}
and
\begin{equation} \label{eq:space-xcoeff}
	\begin{split}
		d_{k_{1}} &:= (k_{1}+1)^{1-\eta_{1}}-k_{1}^{1-\eta_{1}},
		\qquad k_{1} = 0,1,\ldots \\
		\delta_0 &:= d_0, \\
		\delta_{k_{1}} &:= d_{k_{1}} - d_{k_{1}-1} = (k_{1}+1)^{1-\eta_{1}} - 2k_{1}^{1-\eta_{1}} + (k_{1}-1)^{1-\eta_{1}}, \qquad k_{1}=1,\ldots,n_{1}-1.
	\end{split}
\end{equation}
The corresponding error estimate takes the form
\begin{equation*}
	\abs{{}^C\!D_x^{\eta_{1}} u(x,y,t)|_{x=x_i}-{}^C\!D_{L1}^{\eta_{1}} u(x_i,y,t)} \leq
	\frac{1}{2\Gamma(1-\eta_{1})} \left[ \frac{1}{4} + \frac{\eta_{1}}{(1-\eta_{1})(2-\eta_{1})} \right] \Delta x^{2-\eta_{1}} \max_{0\leq x \leq x_i} \abs{u_{xx}(x,t)},
\end{equation*}
with $u(x,\cdot,\cdot)\in C^2(0,1)$ and $x_i\in (0,1)$; {see \cite[page 32]{Sun2020}}.
\vspace{0.3cm}

Now, we proceed to discretize the Caputo fractional derivative with respect to $y$ by employing the $L1$ approximation formula
\begin{equation} \label{eq:Caputo-ydiscre}
	{}^C\!D_{L1}^{\eta_{2}} u(x_i,y_{j},t_{s}) = \frac{1}{\beta_{n_{2}}} \left[ \sigma_0 u_{i,j}^{s} + \sum_{k_{2}=1}^{j-1} \sigma_{j-k_{2}} u_{i,k_{2}}^{s} - c_{j-1} u_{i,0}^{s} \right],
\end{equation}
where
\begin{equation}\label{eq:beta-n2}
	\beta_{n_2} := (\Delta x)^{\eta_{2}} \Gamma(2-\eta_{2}) = \bigg(\frac{1}{n_{2}+1}\bigg)^{\eta_{2}} \Gamma(2-\eta_{2})
\end{equation}
and
\begin{equation} \label{eq:space-ycoeff}
	\begin{split}
		c_{k_{2}} &:= (k_{2}+1)^{1-\eta_{2}}-k_{2}^{1-\eta_{2}},
		\qquad k_{2} = 0,1,\ldots \\
		\sigma_0 &:= c_0, \\
		\sigma_{k_2} &:= c_{k_2} - c_{k_{2}-1} = (k_{2}+1)^{1-\eta_{2}} - 2k_{2}^{1-\eta_{2}} + (k_{2}-1)^{1-\eta_{2}}, \qquad k_{2}=1,\ldots,n_{2}-1.
	\end{split}
\end{equation}
\vspace{0.3cm}

Substituting the fractional derivatives in problem \eqref{eq:quasi-problem} with their respective $L1$ approximations \eqref{eq:tem-discre} and \eqref{eq:Caputo-xdiscre} and \eqref{eq:Caputo-ydiscre}, and omitting the truncation errors, the finite difference scheme is constructed as follows:
\begin{equation*}
	\begin{dcases}
	\frac{1}{\alpha_m} \left[ \gamma_0 u_{i,j}^{s} + \sum_{l=1}^{s-1} \gamma_{s-l} e^{-(n-l)\rho\Delta t} u_{i,j}^{l} - b_{s-1} e^{-s \rho \Delta t} u_{i,j}^{0} \right]
		+ \frac{a_{i,j}}{\beta_{n_{1}}} \left[ \delta_0 u_{i,j}^{s} + \sum_{k_{1}=1}^{i-1} \delta_{i-k_{1}} u_{k_{1},j}^{s} - d_{i-1} u_{0,j}^{s} \right]\\+\frac{a_{i,j}}{\beta_{n_{2}}} \left[ \sigma_0 u_{i,j}^{s} + \sum_{k_{2}=1}^{j-1} \sigma_{j-k_{2}} u_{i,k_{2}}^{s} - c_{j-1} u_{j,0}^{s} \right]
		= q^{s} f_{i,j},
		\\
		u_{0,j}^{s} = u_{n_{1}+1,j}^{s} = 0,\qquad u_{i,0}^{s} = u_{i,n_{2}+1}^{s} = 0,
		\\
		u_{i,j}^{0} = \varphi_{i,j},
		\\
		u_{i,j}^{m} = \phi_{i,j} - \lambda \frac{a_{i,j}}{\beta_{n_{1}}}
		\left[ \delta_0 f_{i,j} + \sum_{k_{1}=1}^{i-1} \delta_{i-k_{1}} f_{k_{1},j} - d_{i-1} f_{0,j} \right] - \lambda \frac{a_{i,j}}{\beta_{n_{2}}}
		\left[ \sigma_0 f_{i,j} + \sum_{k_{2}=1}^{j-1} \sigma_{j-k_{2}} f_{i,k_{2}} - c_{j-1} f_{i,0} \right].
	\end{dcases}
\end{equation*}
By defining the corresponding vectors
\begin{alignat*}{3}
	\mathbf{u}^{s} &:= \begin{bmatrix} u_{i,j}^{s} \end{bmatrix}_{i=1,\ldots,n_{1},\,j=1,\ldots,n_{2}}, & \qquad
	\mathbf{a} &:= \begin{bmatrix} a_{i,j} \end{bmatrix}_{i=1,\ldots,n_{1},\,j=1,\ldots,n_{2}}, & \qquad
	\mathbf{f} &:= \begin{bmatrix} f_{i,j} \end{bmatrix}_{i=1,\ldots,n_{1},\,j=1,\ldots,n_{2}},
	\\
	\boldsymbol{\varphi} &:= \begin{bmatrix} \varphi_{i,j} \end{bmatrix}_{i=1,\ldots,n_{1},\,j=1,\ldots,n_{2}}, & \qquad
	\boldsymbol{\phi}&:= \begin{bmatrix} \phi_{i,j} \end{bmatrix}_{i=1,\ldots,n_{1},\,j=1,\ldots,n_{2}},
\end{alignat*}
Now, let $\mathbf{n}=n_{1}n_{2}$, then we can write the numerical scheme in matrix form as
\begin{equation*}
	\begin{dcases}
		\gamma_0 \mathbf{u}^{s} + \sum_{l=1}^{s-1} \gamma_{s-l} e^{-(s-l)\rho\Delta t} \mathbf{u}^{l}
		+  D_{\mathbf{n}}(a) \left(B_{\mathbf{n}}^{x}+B_{\mathbf{n}}^{y}\right) \mathbf{u}^{s}
		- \alpha_m q^{s} \mathbf{f}
		= b_{s-1} e^{-s\rho\Delta t} \boldsymbol{\varphi},
		& \quad s=1,\ldots,m, \\
		\mathbf{u}^{m} + {\lambda} D_{\mathbf{n}}(a) \left(B_{\mathbf{n}}^{x}+B_{\mathbf{n}}^{y}\right) \mathbf{f} = \boldsymbol{\phi},
	\end{dcases}
\end{equation*}
where $D_{\mathbf{n}}(a)$ is the diagonal sampling matrix of size $M$ associated to $a(x,y)$, while $B_{M}^{x}$ are $B_{M}^{y}$ are defined as follows
\begin{equation}\label{eq:xB}
B_{\mathbf{n}}^{x}=I_{n_{2}} \otimes\frac{\alpha_{m}}{\beta_{n_{1}}} B^{x}_{n_{1}},
\end{equation}
and
\begin{equation}\label{eq:yB}
	B_{\mathbf{n}}^{y}=\frac{\alpha_{m}}{\beta_{n_{2}}}B^{y}_{n_{2}} \otimes I_{n_{1}}.
\end{equation}
Here $B_{n_{1}}$ are $B_{n_{2}}$  are the following lower triangular Toeplitz matrices
\begin{equation*}
	B^{x}_{n_{1}} := \begin{bmatrix}
		\delta_0 & 0 & \cdots & \cdots & 0 \\
		\delta_1 & \delta_0 & 0 & & \vdots \\
		\delta_2 & \delta_1 & \delta_0 & \ddots & \vdots \\
		\vdots & \ddots & \ddots & \ddots& 0  \\
		\delta_{n_{1}-1} & \cdots & \delta_2 & \delta_1 & \delta_0
	\end{bmatrix},
\end{equation*}
and
\begin{equation*}
	B^{y}_{n_{2}} := \begin{bmatrix}
		\sigma_0 & 0 & \cdots & \cdots & 0 \\
		\sigma_1 & \delta_0 & 0 & & \vdots \\
		\sigma_2 & \delta_1 & \delta_0 & \ddots & \vdots \\
		\vdots & \ddots & \ddots & \ddots& 0  \\
		\sigma_{n_{2}-1} & \cdots & \delta_2 & \sigma_1 & \sigma_0
	\end{bmatrix}.
\end{equation*}
By gathering the equations for all the time levels and the regularization equation, we obtain the all-at-once linear system
\begin{equation} \label{eq:lin-syst}
	A_{\widetilde{\mathbf{N}}}\mathbf{v}_\lambda = \mathbf{z}_\lambda,
\end{equation}
where $\widetilde{\mathbf{N}} = \widetilde{\mathbf{N}}(m,\mathbf{n}) := (m+1)\nu(\mathbf{n})$ and
\begin{gather*}
	\renewcommand{\arraystretch}{1.5}
	A_{\widetilde{\mathbf{N}}} := \left[ \begin{array}{cccc|c}
		\gamma_0 I_{\mathbf{n}} +  G_\mathbf{n} & \mathrm{O}_\mathbf{n} & \cdots & \mathrm{O}_\mathbf{n} & -\alpha_m q^{1} I_\mathbf{n} \\
		\gamma_1 e^{-\rho \Delta t} I_\mathbf{n} & \gamma_0 I_\mathbf{n} +  G_\mathbf{n} & \ddots & \vdots & -\alpha_m q^{2} I_\mathbf{n} \\
		\vdots & \ddots & \ddots & \mathrm{O}_\mathbf{n} & \vdots \\
		\gamma_{m-1} e^{-(m-1)\rho\Delta t} I_\mathbf{n} & \cdots & \gamma_1 e^{-\rho \Delta t} I_\mathbf{n} & \gamma_0 I_\mathbf{n} +  G_\mathbf{n} & -\alpha_m q^{m} I_\mathbf{n} \\
		\hline
		\mathrm{O}_M & \cdots & \mathrm{O}_\mathbf{n} & I_\mathbf{n} & \frac{{\lambda}}{\alpha_{m}} G_\mathbf{n}
	\end{array} \right]
	\in \C^{\widetilde{\mathbf{N}}\times \widetilde{\mathbf{N}}},
	\\
	\renewcommand{\arraystretch}{1.5}
	\mathbf{v}_{\lambda} :=
	\left[ \begin{array}{c}
		\mathbf{u}^{1} \\ \mathbf{u}^{2} \\ \vdots \\ \mathbf{u}^{m} \\ \hline \mathbf{f}
	\end{array} \right]
	\in\C^{\widetilde{\mathbf{N}}},
	\qquad
	\mathbf{z}_{\lambda} :=
	\left[ \begin{array}{c}
		b_0 e^{-\rho\Delta t} \boldsymbol{\varphi} \\
		b_1 e^{-2\rho\Delta t} \boldsymbol{\varphi} \\
		\vdots \\
		b_{m-1} e^{-m\rho\Delta t} \boldsymbol{\varphi} \\
		\hline \boldsymbol{\phi}
	\end{array} \right]
	\in\C^{\widetilde{\mathbf{N}}},
\end{gather*}
in which $I_\mathbf{n}$ is the $\mathbf{n}\times\mathbf{n}$ identity matrix, $\mathrm{O}_\mathbf{n}$ is the $\mathbf{n}\times \mathbf{n}$ zero matrix, and
\begin{equation} \label{eq:Gm-def}
	G_\mathbf{n} := D_\mathbf{n}(a) \left(B_{\mathbf{n}}^{x}+B_{\mathbf{n}}^{y}\right).
\end{equation}
Moreover, setting
\begin{equation*}
	U_m := \begin{bmatrix}
		\gamma_0 & 0 & \cdots & 0 \\
		\gamma_1 e^{-\rho\Delta t} & \gamma_0 & \ddots & \vdots \\
		\vdots & \ddots & \ddots & 0 \\
		\gamma_{m-1} e^{-\rho(m-1)\Delta t} & \cdots & \gamma_1 e^{-\rho\Delta t} & \gamma_0
	\end{bmatrix},
\end{equation*}
the coefficient matrix can be expressed into the compact form
\begin{equation} \label{eq:coeff-mat}
	\renewcommand{\arraystretch}{1.5}
	A_{\widetilde{\mathbf{N}}} = \left[ \begin{array}{c|c}
		U_m \otimes I_\mathbf{n} + I_m\otimes G_\mathbf{n} & -\alpha_m \mathbf{q}\otimes I_\mathbf{n} \\
		\hline
		\mathbf{e}_{m}^T \otimes I_\mathbf{n} & \frac{{\lambda}}{\alpha_{m}} G_\mathbf{n}
	\end{array} \right],
\end{equation}
in which
\begin{equation*}
	\mathbf{e}_m := \begin{bmatrix}
		0 \\ \vdots \\ 0 \\ 1
	\end{bmatrix} \in\R^m,
	\qquad
	\mathbf{q} := \begin{bmatrix}
		q^{1} \\ q^{1} \\ \vdots \\ q^{m}
	\end{bmatrix} \in\R^m.
\end{equation*}
\vspace{0.3cm}

\section{GLT analysis} \label{sec:glt-analysis}
This section presents the spectral analysis of the coefficient matrix sequence via GLT theory to describe its asymptotic behavior. We first report the symbols of $\{B^{x}_{n}\}_{n}$ and the time block $\{U_m\}_{m}$ \cite{Ilyas2026}, then analyze the space block $\{G_{\mathbf{n}}\}_{\mathbf{n}}$, and finally consider the full coefficient matrix sequence.

\begin{theorem} \cite{Ilyas2026}\label{thm:GLT-Bm}
	There exist continuous $2\pi$-periodic functions $g_{\eta_{1}}$ and $g_{\eta_{2}}$ such that
	\begin{align*}
		\{B^{x}_{n_{1}}\}_{n_{1}} \sim_{\glt,\sigma}& g_{\eta_{1}}(\theta_{1}),
		\qquad \theta_{1}\in [-\pi,\pi],\\
		\{B^{y}_{n_{2}}\}_{n_{2}} \sim_{\glt,\sigma}& g_{\eta_{2}}(\theta_{2}),
		\qquad \theta_{2}\in [-\pi,\pi].
	\end{align*}
	Moreover,
	\begin{align*}
		\{B^{x}_{n_{1}}\}_{n_{1}} \sim_\lambda& \delta_0,\\
			\{B^{y}_{n_{2}}\}_{n_{2}} \sim_\lambda& \sigma_0.
	\end{align*}
\end{theorem}
\begin{theorem} \cite{Ilyas2026}\label{thm:GLT-Un}
	There exists a continuous $2\pi$-periodic function $h_\xi$ such that
	\begin{equation*}
		\{U_m\}_m \sim_{\glt,\sigma} h_\xi(\theta),
		\qquad \theta\in [-\pi,\pi].
	\end{equation*}
	Moreover,
	\begin{equation*}
		\{U_m\}_m \sim_\lambda \gamma_0.
	\end{equation*}
\end{theorem}

\begin{corollary}\label{cor:distr}
	If $a:[0,1]^2\to\R$ is continuous a.e. and if
	\begin{equation*}
		\frac{\alpha_m}{\beta_{n_{1}}} \to \nu_{1},\quad \frac{\alpha_m}{\beta_{n_{2}}} \to \nu_{2}, \qquad \text{as } m,n_{1},n_{2}\to\infty,
	\end{equation*}
	with $\nu_{1}\neq 0$ or $\nu_{2}\neq 0$. Then, we have
	\begin{equation*}
		\{G_\mathbf{n}\}_\mathbf{n}  \sim_{\glt,\sigma} a(x,y) \left(\nu_{1}g_{\eta_{1}}(\theta_{1})+\nu_{2}g_{\eta_{2}}(\theta_{2})\right),
		\qquad (x,y)\in [0,1]^{2}, \;(\theta_{1},\, \theta_{2})\in [-\pi,\pi]^{2},
	\end{equation*}
	with $g_{\eta_{1}}$ and $g_{\eta_{2}}$ as in Theorem \ref{thm:GLT-Bm}. Moreover,
	\begin{equation*}
		\{G_\mathbf{n}\}_\mathbf{n}  \sim_\lambda  a(x,y)\left(\delta_0+\sigma_{0}\right),
		\qquad (x,y)\in [0,1]^{2}.
	\end{equation*}
\end{corollary}
\begin{proof}
	From the definition in \eqref{eq:Gm-def}, we have
	\begin{equation*}
		G_\mathbf{n} := D_\mathbf{n}(a) \left(B_{\mathbf{n}}^{x}+B_{\mathbf{n}}^{y}\right).
	\end{equation*}
	Thus, by GLT \ref{glt:symbols}, we deduce that
	\begin{equation*}
		\{D_\mathbf{n}(a)\}_{\mathbf{n}} \sim_\glt a(x,y), \qquad (x,y)\in [0,1]^{2}.
	\end{equation*}
	Then, using \eqref{eq:xB} and assuming $\frac{\alpha_m}{\beta_{n_{1}}}\to \nu_{1}$, by GLT \ref{glt:symbols}, and GLT \ref{glt:kron}, we obtain
	\begin{align*}
		\{B^{x}_{\mathbf{n}}\}_{\mathbf{n}}=\left\{I_{n_{2}} \otimes\frac{\alpha_{m}}{\beta_{n_{1}}} B^{x}_{n_{1}}\right\}_{\mathbf{n}}=\left\{T_{n_{2}}(1) \otimes\frac{\alpha_{m}}{\beta_{n_{1}}} T_{n_{1}(g_{\eta_{1}})}\right\}_{\mathbf{n}}  \sim_{\glt}& \,\nu_{1}g_{\eta_{1}}(\theta_{1})\qquad \theta_{1}\in [-\pi,\pi].
	\end{align*}
	Similarly, using \eqref{eq:yB} and assuming $\frac{\alpha_m}{\beta_{n_{2}}}\to \nu_{2}$, by GLT \ref{glt:symbols}, and GLT \ref{glt:kron}, we obtain
	\begin{align*}
		\{B^{y}_{\mathbf{n}}\}_{\mathbf{n}}=\left\{\frac{\alpha_{m}}{\beta_{n_{2}}} B^{y}_{n_{2}}  \otimes I_{n_{1}}\right\}_{\mathbf{n}}=\left\{\frac{\alpha_{m}}{\beta_{n_{2}}} T_{n_{2}(g_{\eta_{2}})} \otimes T_{n_{1}}(1)\right\}_{\mathbf{n}}  \sim_{\glt}& \,\nu_{2}g_{\eta_{2}}(\theta_{2}),
		\qquad  \theta_{2}\in [-\pi,\pi].
	\end{align*}
	Invoking the $\ast$-algebra properties detailed in GLT \ref{glt:algebra}, we obtain that
	\begin{equation*}
	\{B^{x}_{\mathbf{n}}+B^{y}_{\mathbf{n}}\}_{\mathbf{n}}   \sim_{\glt} \nu_{1}g_{\eta_{1}}(\theta_{1})+\nu_{2}g_{\eta_{2}}(\theta_{2}),
	\qquad  \;(\theta_{1},\, \theta_{2})\in [-\pi,\pi]^{2}.
\end{equation*}
Using the product property of GLT \ref{glt:algebra} applied to the sequences $\{D_\mathbf{n}(a)\}_\mathbf{n}$ and $\{B^{x}_\mathbf{n}+B^{y}_\mathbf{n}\}_\mathbf{n}$, we conclude that
\begin{equation*}
	\{G_\mathbf{n}\}_\mathbf{n}  \sim_{\glt} a(x,y) \left(\nu_{1}g_{\eta_{1}}(\theta_{1})+\nu_{2}g_{\eta_{2}}(\theta_{2})\right),
	\qquad (x,y)\in [0,1]^{2}, \;(\theta_{1},\, \theta_{2})\in [-\pi,\pi]^{2},
\end{equation*}
	By GLT \ref{glt:distrib}, we immediately obtain
		\begin{equation*}
		\{G_\mathbf{n}\}_\mathbf{n}  \sim_{\sigma} a(x,y) \left(\nu_{1}g_{\eta_{1}}(\theta_{1})+\nu_{2}g_{\eta_{2}}(\theta_{2})\right),
		\qquad (x,y)\in [0,1]^{2}, \;(\theta_{1},\, \theta_{2})\in [-\pi,\pi]^{2},
	\end{equation*}
	Finally, we observe that $B^{x}_\mathbf{n}$ and $B^{y}_\mathbf{n}$ are lower triangular matrices. Hence $\left(B^{x}_\mathbf{n} + B^{y}_\mathbf{n}\right)$ is lower triangular, and since $D_{\mathbf{n}}(a)$ is diagonal,
	\begin{equation*}
		G_\mathbf{n} := D_\mathbf{n}(a) \left(B_{\mathbf{n}}^{x}+B_{\mathbf{n}}^{y}\right),
	\end{equation*}
	is also lower diagonal. Therefore, the eigenvalues of $G_\mathbf{n}$ coincide with its diagonal entries.
	Since the main diagonal entries of $B^{x}_\mathbf{n}$ and $B^{y}_\mathbf{n}$ are $\delta_{0}$ and $\sigma_{0}$, the diagonal entries of $G_\mathbf{n}$ are $\{a(x_{i},y_{j})\left(\delta_0+\sigma_{0}\right)\}_{i,j=1}^{n_{1},n_{2}}$.
	Hence
	\begin{equation*}
		\{G_\mathbf{n}\}_\mathbf{n}  \sim_\lambda  a(x,y)\left(\delta_0+\sigma_{0}\right),
		\qquad (x,y)\in [0,1]^{2}.
	\end{equation*}
\end{proof}

We are now in a position to compute the GLT symbol of the sequence associated with the full coefficient matrix $A_{\widetilde{\mathbf{N}}}$. To ensure meaningful results, we assume a mild relation among the rates at which the discretization parameters tend to zero.

\begin{theorem} \label{thm:coeff-mat-sym}
	With $A_{\widetilde{\mathbf{N}}}$ defined in \eqref{eq:coeff-mat}, assume that
	\begin{equation*}
		\frac{\alpha_m}{\beta_{n_{1}}} \to \nu_{1},\quad \frac{\alpha_m}{\beta_{n_{2}}} \to \nu_{2}, \qquad \text{as } m,n_{1},n_{2}\to\infty,
	\end{equation*}
	with $\nu_{1}\neq 0$ or $\nu_{2}\neq 0$. Then, it holds
	\begin{equation*}
		\{A_{\widetilde{\mathbf{N}}}\}_{\widetilde{\mathbf{N}}} \sim_{\glt,\sigma} h_\xi(\theta_1) + a(x,y) \left(\nu_{1}g_{\eta_{1}}(\theta_{2})+\nu_{2}g_{\eta_{2}}(\theta_{3})\right), \quad (\theta_1,\theta_2,\theta_{3})\in [-\pi,\pi]^3, \;\;(x,y)\in [0,1]^{2},
	\end{equation*}
	where $h_\xi$ is given in Theorems \ref{thm:GLT-Bm} and $g_{\eta_{1}}$ and $g_{\eta_{2}}$  are given in Theorems \ref{thm:GLT-Un}.
\end{theorem}

\begin{proof}
We proceed by decomposing \eqref{eq:coeff-mat} as follows
	\begin{equation*}
		\renewcommand{\arraystretch}{1.5}
		A_{\widetilde{\mathbf{N}}} = \underbrace{ \left[ \begin{array}{cc}
				U_m \otimes I_\mathbf{n} & \mathrm{O}_{m\mathbf{n}\times \mathbf{n}} \\
				\mathrm{O}_{\mathbf{n}\times m\mathbf{n}} & \mathrm{O}_\mathbf{n}
			\end{array} \right] }_{=:\, C_{\widetilde{\mathbf{N}}}^{(1)}}
		+ \underbrace{ \left[ \begin{array}{cc}
				I_m\otimes  G_\mathbf{n} & \mathrm{O}_{m\mathbf{n}\times \mathbf{n}} \\
				\mathrm{O}_{\mathbf{n}\times m\mathbf{n}} & \frac{{\lambda}}{\alpha_m} G_\mathbf{n}
			\end{array} \right] }_{=:\, C_{\widetilde{\mathbf{N}}}^{(2)}}
		+ \underbrace{ \left[ \begin{array}{cc}
				\mathrm{O}_{m\mathbf{n}} & -\alpha_m\mathbf{q}\otimes I_\mathbf{n} \\
				\mathbf{e}_m^T \otimes I_\mathbf{n} & \mathrm{O}_\mathbf{n}
			\end{array} \right] }_{=:\, C_{\widetilde{\mathbf{N}}}^{(3)}}.
	\end{equation*}
	Thus, we analyze each component separately. For the first component, we have
	\begin{equation*}
		\renewcommand{\arraystretch}{1.3}
		C_{\widetilde{\mathbf{N}}}^{(1)}
		= U_{m+1} \otimes I_\mathbf{n}
		- \underbrace{\left[ \begin{array}{c|c}
				\mathrm{O}_{m\mathbf{n}} & \mathrm{O}_{m\mathbf{n}\times \mathbf{n}} \\
				\hline
				\begin{matrix}
					\gamma_m e^{-\rho m\Delta t} I_\mathbf{n} & \cdots & \gamma_1 e^{-\rho\Delta t} I_\mathbf{n}
				\end{matrix} & \gamma_0 I_m
			\end{array} \right]}_{=: R^{(1)}_{\widetilde{\mathbf{N}}}},
	\end{equation*}
	where $\rk R^{(1)}_{\widetilde{\mathbf{N}}} \leq \nu(\mathbf{n}) = o({\widetilde{\mathbf{N}}})$ as $m,\mathbf{n}\to\infty$. According to Theorem \ref{lemma:zero-distr}, the sequence $\{R^{(1)}_{\widetilde{\mathbf{N}}}\}_{\widetilde{\mathbf{N}}}$ is zero-distributed; hence, by GLT \ref{glt:symbols}, we deduce that
	\begin{equation*}
		\{R^{(1)}_{\widetilde{\mathbf{N}}}\}_{\widetilde{\mathbf{N}}}\sim_\glt 0.
	\end{equation*}
	Moreover, exploiting the identity $U_{m+1} \otimes I_\mathbf{n}= T_{m+1}(h_\xi) \otimes T_\mathbf{n}(1)$ and applying the axioms GLT \ref{glt:symbols}--\ref{glt:kron} yields
	\begin{equation} \label{eq:S1}
		\big\{C_{\widetilde{\mathbf{N}}}^{(1)}\big\}_{\widetilde{\mathbf{N}}} \sim_\glt h_\xi(\theta_1).
	\end{equation}
	Regarding $C_{\widetilde{\mathbf{N}}}^{(2)}$, we observe that
	\begin{equation*}
		C_{\widetilde{\mathbf{N}}}^{(2)} = \left[ \begin{array}{cc}
			I_m\otimes  G_\mathbf{n} & \mathrm{O}_{m\mathbf{n}\times\mathbf{n}} \\
			\mathrm{O}_{\mathbf{n}\times m\mathbf{n}} &  G_\mathbf{n} + \left(\frac{\lambda}{\alpha_m}-1\right) G_\mathbf{n}
		\end{array} \right]
		= I_{m+1} \otimes   G_\mathbf{n} +
		\underbrace{\left[ \begin{array}{cc}
				\mathrm{O}_{m\mathbf{n}} & \mathrm{O}_{m\mathbf{n}\times \mathbf{n}} \\
				\mathrm{O}_{\mathbf{n}\times m\mathbf{n}} & \left(\frac{\lambda}{\alpha_m}-1\right) G_\mathbf{n}
			\end{array} \right]}_{=: R_{\widetilde{\mathbf{N}}}^{(2)}},
	\end{equation*}
	where, once again, $\rk R_{\widetilde{\mathbf{N}}}^{(2)} = \nu(\mathbf{n}) = o(\widetilde{\mathbf{N}})$ as $m,\mathbf{n}\to\infty$. By Theorem \ref{lemma:zero-distr}, the sequence $\{R^{(2)}_{\widetilde{\mathbf{N}}}\}_{\widetilde{\mathbf{N}}}$ is zero-distributed; therefore, by GLT \ref{glt:symbols}, it follows that
	\begin{equation*}
		\{R_{\widetilde{\mathbf{N}}}^{(2)}\}_{\widetilde{\mathbf{N}}}\sim_\glt 0.
	\end{equation*}
	Invoking the assumptions $\frac{\alpha_m}{\beta_{n_1}}\to \nu_{1}$ and $\frac{\alpha_m}{\beta_{n_2}}\to \nu_{2}$ together with GLT \ref{glt:symbols} yields
	\begin{equation*}
		\left\{I_{m+1} \otimes G_\mathbf{n}\right\}_{\widetilde{\mathbf{N}}} =\left\{I_{m+1} \otimes  D_{\mathbf{n}}(a)\left(B_{\mathbf{n}}^{x}+B_{\mathbf{n}}^{y}\right)\right\}_{\widetilde{\mathbf{N}}}= \left\{ T_{m+1}(1) \otimes  D_\mathbf{n}(a)\left(\nu_{1}T_\mathbf{n}(g_{\eta_{1}})+\nu_{2}T_\mathbf{n}(g_{\eta_{2}})\right) \right\}_{\widetilde{\mathbf{N}}} .
	\end{equation*}
	Then, by GLT \ref{glt:kron}, we obtain
	\begin{equation*}
		\left\{I_{m+1} \otimes G_\mathbf{n}\right\}_{\widetilde{\mathbf{N}}}\sim_\glt a(x,y) \left(\nu_{1}g_{\eta_{1}}(\theta_{2})+\nu_{2}g_{\eta_{2}}(\theta_{3})\right).
	\end{equation*}
	By GLT \ref{glt:algebra}, we conclude that
	\begin{equation} \label{eq:S2}
		\big\{C_{\widetilde{\mathbf{N}}}^{(2)}\big\}_{\widetilde{\mathbf{N}}} \sim_\glt a(x,y) \left(\nu_{1}g_{\eta_{1}}(\theta_{2})+\nu_{2}g_{\eta_{2}}(\theta_{3})\right).
	\end{equation}
	We conclude by observing that
	\begin{equation*}
		C_{\widetilde{\mathbf{N}}}^{(3)} = \left[ \begin{array}{cc}
			\mathrm{O}_{m\mathbf{n}} & -\alpha_m\mathbf{q}\otimes I_\mathbf{n} \\
			\mathrm{O}_{\mathbf{n}\times m\mathbf{n}} & \mathrm{O}_\mathbf{n}
		\end{array} \right]
		+ \left[ \begin{array}{cc}
			\mathrm{O}_{m\mathbf{n}} & \mathrm{O}_{m\mathbf{n}\times \mathbf{n}} \\
			\mathbf{e}_{m}^T \otimes I_\mathbf{n} & \mathrm{O}_\mathbf{n}
		\end{array} \right],
	\end{equation*}
	where each component has rank at most $\mathbf{n}=o({\widetilde{\mathbf{N}}})$. Hence, we conclude that
	\begin{equation} \label{eq:S3}
		\big\{C_{\widetilde{\mathbf{N}}}^{(3)}\big\}_{\widetilde{\mathbf{N}}} \sim_\glt 0.
	\end{equation}
	By combining the contributions in \eqref{eq:S1}, \eqref{eq:S2}, and \eqref{eq:S3}, and invoking GLT \ref{glt:algebra}, we obtain the GLT symbol of the matrix sequence ${A_{\widetilde{\mathbf{N}}}}_{\widetilde{\mathbf{N}}}$
	\begin{equation*}
		\{A_{\widetilde{\mathbf{N}}}\}_{\widetilde{\mathbf{N}}} \sim_{\glt} h_\xi(\theta_1) + a(x,y) \left(\nu_{1}g_{\eta_{1}}(\theta_{2})+\nu_{2}g_{\eta_{2}}(\theta_{3})\right), \quad (\theta_1,\theta_2,\theta_{3})\in [-\pi,\pi]^3, \;\;(x,y)\in [0,1]^{2}.
	\end{equation*}
By GLT \ref{glt:distrib}, the singular value distribution is
	\begin{equation*}
	\{A_{\widetilde{\mathbf{N}}}\}_{\widetilde{\mathbf{N}}} \sim_{\sigma} h_\xi(\theta_1) + a(x,y) \left(\nu_{1}g_{\eta_{1}}(\theta_{2})+\nu_{2}g_{\eta_{2}}(\theta_{3})\right), \quad (\theta_1,\theta_2,\theta_{3})\in [-\pi,\pi]^3, \;\;(x,y)\in [0,1]^{2},
\end{equation*}
\end{proof}

\begin{Remark} \label{remark:dimensional parameters}
The assumption $\nu_{1}\neq 0$ or $\nu_{2}\neq 0$ used in Corollary \ref{cor:distr} and in Theorem \ref{thm:coeff-mat-sym} deserves a bit of attention and few comments since it is used also in the study of the preconditioning; see Theorem \ref{thm:precond-sym}. In fact, if we assume $n_1\sim n_2$ i.e.
$cn_1\le n_2 \le Cn_1$ with $c, C$ universal constants independent of $n_1,n_2$, the we find the following cases:
\begin{itemize}
\item when $\eta_{1}> \eta_2$, taking into account the expression in (\ref{eq:beta-n1}), (\ref{eq:beta-n2}), the given assumption implies
$\nu_{1}> 0$ and $\nu_{2}= 0$;
\item when $\eta_{2}> \eta_1$, taking into account the expression in (\ref{eq:beta-n1}), (\ref{eq:beta-n2}), the given assumption implies
$\nu_{1}= 0$ and $\nu_{2}> 0$;
\item when $\eta_{2}= \eta_1$, taking into account the expression in (\ref{eq:beta-n1}), (\ref{eq:beta-n2}), the given assumption implies
$\nu_{1}> 0$ and $\nu_{2}> 0$ with $\nu_{1}=\nu_{2}> 0$ if $n_2=n_1+o(n_1)$;
\end{itemize}
Finally, it is worth observing that when $n_1\sim n_2$ is violated, the picture becomes more complicate, but it is also a degenerate case since it implies that one space direction is approximated with much more precision than the other.
\end{Remark}

\section{GLT preconditioning} \label{sec:precond}

The current section is divided into two parts. The first contains the proposal and the second the related spectral analysis.

\subsection{Block Triangular preconditioning}
In this subsection, we propose block triangular preconditioners $P_{\widetilde{\mathbf{N}}}$ for the linear system \eqref{eq:lin-syst}, aiming to cluster the eigenvalues of $\{ P_{\widetilde{\mathbf{N}}}^{-1} A_{\widetilde{\mathbf{N}}} \}_{\widetilde{\mathbf{N}}} $ around $1$. We define
\begin{equation} \label{eq:precond}
	\renewcommand{\arraystretch}{1.5}
	P_{\widetilde{\mathbf{N}}} := \left[ \begin{array}{c|c}
		U_m \otimes I_\mathbf{n} + I_m\otimes G_\mathbf{n} & \mathrm{O}_{\mathbf{n}m \times \mathbf{n}} \\
		\hline
		\mathbf{e}_m^T \otimes I_\mathbf{n} & \frac{{\lambda}}{\alpha_{m}} G_\mathbf{n}
	\end{array} \right].
\end{equation}
Note that both $U_m$ and $G_\mathbf{n}$ are invertible, since they are lower triangular matrices with nonzero diagonal entries. The same holds for $U_m \otimes I_\mathbf{n} + I_m \otimes G_\mathbf{n}$, which is also block lower triangular. Moreover, due to the block triangular structure, the eigenvalues of $P_{\widetilde{\mathbf{N}}}$ coincide with those of its diagonal blocks. Therefore, $P_{\widetilde{\mathbf{N}}}$ is invertible.

The linear systems with matrix $P_{\widetilde{\mathbf{N}}}$ that have to be solved when applying the preconditioned GMRES to the original linear systems can be solved by standard block forward elimination. However we have to specify how the linear systems having as coefficient matrices the diagonal blocks of $P_{\widetilde{\mathbf{N}}}$ are solved. Indeed the diagonal blocks have the expression $\gamma_0 I_{\mathbf{n}} + D_{\mathbf{n}}(a) B_{\mathbf{n}}$ and for them we can use several standard preconditiones such as $\gamma_0 I_{\mathbf{n}} + \hat a B_{\mathbf{n}}$ which is of Toeplitz lower triangular form, with $\hat a$ being the integral average of the function $a$, or the standard circulant preconditioner constructed as $\gamma_0 I_{\mathbf{n}} + \hat a S(B_{\mathbf{n}})$, $S(B_{\mathbf{n}})$ being a Strang-like circulant approximation of $B_{\mathbf{n}}$, constructed by preserving the first column of $B_{\mathbf{n}}$. The computational features of these choice are discussed at the end of Section \ref{sec:num}.

\subsection{GLT preconditioning analysis}
The following theorem establishes that $P_{\widetilde{\mathbf{N}}}$ has the same GLT symbol, and consequently the same singular value distribution, as the original coefficient matrix sequence.

\begin{theorem} \label{thm:precond-sym}
	With $P_{\widetilde{\mathbf{N}}}$ defined in \eqref{eq:precond}, assume that
	\begin{equation*}
		\frac{\alpha_m}{\beta_{n_{1}}} \to \nu_{1},\quad \frac{\alpha_m}{\beta_{n_{2}}} \to \nu_{2}, \qquad \text{as } m,n_{1},n_{2}\to\infty,
	\end{equation*}
	with $\nu_{1}\neq 0$ or $\nu_{2}\neq 0$. Then, it holds
	\begin{equation*}
		\{P_{\widetilde{\mathbf{N}}}\}_{\widetilde{\mathbf{N}}} \sim_{\glt,\sigma} h_\xi(\theta_1) + a(x,y) \left(\nu_{1}g_{\eta_{1}}(\theta_{2})+\nu_{2}g_{\eta_{2}}(\theta_{3})\right), \quad (\theta_1,\theta_2,\theta_{3})\in [-\pi,\pi]^3, \;\;(x,y)\in [0,1]^{2},
	\end{equation*}
	where $h_\xi$ is given in Theorems \ref{thm:GLT-Bm} and $g_{\eta_{1}}$ and $g_{\eta_{2}}$  are given in Theorems \ref{thm:GLT-Un}.
\end{theorem}

\begin{proof}
	The result follows from the same GLT arguments used in the proof of Theorem \ref{thm:coeff-mat-sym}.
\end{proof}

The GLT symbol and the singular value distribution of the preconditioned matrix sequence follow as a corollary of Theorems \ref{thm:coeff-mat-sym} and \ref{thm:precond-sym}, by exploiting the algebra properties in GLT \ref{glt:algebra}. However, they can also be derived in a more general setting, without assuming specific conditions on the discretization parameters, as shown in the following theorem. Moreover, the eigenvalue distribution can be characterized.

\begin{theorem} 
	With $A_{\widetilde{\mathbf{N}}}$ as in \eqref{eq:coeff-mat} and $P_{\widetilde{\mathbf{N}}}$ as in \eqref{eq:precond}, it holds
	\begin{equation*}
		\{P_{\widetilde{\mathbf{N}}}^{-1}A_{\widetilde{\mathbf{N}}}\}_{\widetilde{\mathbf{N}}} \sim_{\glt,\sigma} 1.
	\end{equation*}
	Moreover,
	\begin{equation*}
		\{L_{\widetilde{\mathbf{N}}}^{-1}A_{\widetilde{\mathbf{N}}}\}_{\widetilde{\mathbf{N}}} \sim_\lambda 1.
	\end{equation*}
\end{theorem}

\begin{proof}
	Define $R_{\widetilde{\mathbf{N}}} := A_{\widetilde{\mathbf{N}}} - P_{\widetilde{\mathbf{N}}}$. Then,
	\begin{equation*}
		P_{\widetilde{\mathbf{N}}}^{-1} A_{\widetilde{\mathbf{N}}} = I_{\widetilde{\mathbf{N}}} + P_{\widetilde{\mathbf{N}}}^{-1}R_{\widetilde{\mathbf{N}}}.
	\end{equation*}
	Since $\rk R_{\widetilde{\mathbf{N}}}\leq \nu(\mathbf{n})$, it follows that $\rk(P_{\widetilde{\mathbf{N}}}^{-1}R_{\widetilde{\mathbf{N}}}) \leq \nu(\mathbf{n}) = o({\widetilde{\mathbf{N}}})$ as $m,\mathbf{n}\to\infty$. By Theorem \ref{lemma:zero-distr}, the sequence $\{P_{\widetilde{\mathbf{N}}}^{-1}R_{\widetilde{\mathbf{N}}}\}_{\widetilde{\mathbf{N}}}$ is zero-distributed; therefore, by GLT \ref{glt:symbols}, it follows that
	\begin{equation*}
		\{P_{\widetilde{\mathbf{N}}}^{-1}R_{\widetilde{\mathbf{N}}}\}_{\widetilde{\mathbf{N}}} \sim_\glt 0.
	\end{equation*}
	Hence, by exploiting part 2 of axiom GLT \ref{glt:algebra} and by the fact that $\{I_{\widetilde{\mathbf{N}}}\}_{\widetilde{\mathbf{N}}} \sim_\glt 1$, we infer
	\begin{equation*}
		\{P_{\widetilde{\mathbf{N}}}^{-1} A_{\widetilde{\mathbf{N}}}\}_{\widetilde{\mathbf{N}}} \sim_\glt 1.
	\end{equation*}
	Thus, by GLT \ref{glt:distrib}, we immediately obtain
	\begin{equation*}
		\{P_{\widetilde{\mathbf{N}}}^{-1} A_{\widetilde{\mathbf{N}}}\}_{\widetilde{\mathbf{N}}} \sim_\sigma 1.
	\end{equation*}
	Turning to the eigenvalue distribution, we observe that
	\begin{equation*}
		\lambda_n (P_{\widetilde{\mathbf{N}}}^{ -1 } A_{\widetilde{\mathbf{N}}} ) = 1 + \lambda_n (P_{\widetilde{\mathbf{N}}}^{-1} R_{\widetilde{\mathbf{N}}}),
		\qquad n=1,\ldots,{\widetilde{\mathbf{N}}},
	\end{equation*}
	therefore
	\begin{equation*}
		\# \left\{ n=1,\ldots,{\widetilde{\mathbf{N}}}  \;\colon\; \lambda_n (P_{\widetilde{\mathbf{N}}}^{-1} A_{\widetilde{\mathbf{N}}}) \neq 1 \right\} \leq \mathbf{n} = o({\widetilde{\mathbf{N}}}).
	\end{equation*}
	Thus, the eigenvalues of $\{ P_{\widetilde{\mathbf{N}}}^{-1} A_{\widetilde{\mathbf{N}}} \}_{\widetilde{\mathbf{N}}}$ are weakly clustered at $1$ in the sense of Definition \ref{clustering}, which, by Remark \ref{remark:distrib-cluster}, is equivalent to spectral distribution as $1$.
\end{proof}

\begin{Remark} \label{remark:future directions}
We emphasize that the given sequence of theoretical steps for studying the distribution and clustering in the eigenvalue and singular value sense can be adapted with little effort to a great variety of different situations. Here we mention two of them, one concerning a generalization of problem \eqref{eq:problem} and one regarding other classes of discretizations for the regularized FDEs in \eqref{eq:quasi-problem}.
\begin{enumerate}
	\item If the square $[0,1]\times [0,1]$ is replaced by a bounded $\Omega\subset \mathbb{R}^2$, with $\Omega$ simply Peano-Jordan measurable, then the domain $\Omega_T$ in \eqref{eq:problem}--\eqref{eq:quasi-problem} is replaced by $\Omega\times (0,T)$. Hence, simply adapting the steps in Section \ref{sec:glt-analysis} and Section \ref{sec:precond}, the same conclusions are obtained, using the reduced GLT theory \cite{reduced} in place of the classical one; for the reduced GLT theory refer to \cite{serra2003}[pp. 395–399] and \cite{serra2006}[Section 3.1.4] for an example, the initial proposal, the related terminology, and \cite{reduced} for a complete theoretical development.
	\item If the discretization formulae used in Section \ref{sec:discre} are replaced by high order finite elements in space of order $p>1$ or by the isogeometric analysis with degree $p$ and regularity $k<p-1$, then we would end up with block multilevel matrix sequences, where with reference to Section \ref{sec:prelim}, the dimensionality $r$ of the GLT symbol would by either $p^d$ or $(p-k)^d$ with $d=2$. Again the analysis given here can be performed almost verbatim with minimal technical changes, thanks to the uniform structure of the GLT axioms.
\end{enumerate}
Of course the directions indicated in 1) and 2) can be combined and even extended to higher space dimensionality with $\Omega\subset \mathbb{R}^d$, $d\ge 2$, so having the potential of providing a very general picture.
\end{Remark}

\section{Numerical results and study of the computational cost}\label{sec:num}
We evaluate the efficiency of the proposed preconditioner by measuring the number of GMRES iterations required for convergence.
Let us consider the two-dimensional problem on the domain
\(\Omega = (0,1)^2\) and \(t \in (0,T)\). We set
\[
T = 1, \qquad a(x,y) = x + y, \qquad q(t) = t^2, \qquad \varphi(x,y) = 0.
\]
To construct the noisy final data, we select the source function
\[
f(x,y) = x y \sin(\pi x)\sin(\pi y).
\]
and solve the linear system arising from the corresponding direct problem
\begin{equation*}
	\begin{dcases} {}^{T}D_{t}^{\xi,\rho}u(x,y,t)+ a(x,y)\left({}^{C}D_{x}^{\eta_{1}}+{}^{C}D_{y}^{\eta_{2}}\right)u(x,y,t)=q(t)f(x,y), & (x,y,t) \in \Omega_{T},\\	u(0,y,t)=0=u(1,y,t), & y \in (0,1), \quad t\in(0,T),\\
		u(x,0,t)=0=u(x,1,t), & x \in (0,1), \quad t\in(0,T),\\
		u(x,y,0) = \varphi(x,y), & x \in (0,1), \quad y \in (0,1).\end{dcases}\end{equation*}
With the same discretization technique adopted in Section \ref{sec:discre}, we obtain the linear system
\begin{equation*}
	\tilde{A}_{m,\mathbf{n}}\tilde{\mathbf{u}} = \tilde{\mathbf{b}},
\end{equation*}
where
\begin{gather*}
	\renewcommand{\arraystretch}{1.5}
	\tilde{A}_{m,\mathbf{n}} := \left[ \begin{array}{cccc}
		\gamma_0 I_\mathbf{n} +  G_\mathbf{n} & \mathrm{O}_\mathbf{n} & \cdots & \mathrm{O}_\mathbf{n} \\
		\gamma_1 e^{-\rho \Delta t} I_\mathbf{n} & \gamma_0 I_\mathbf{n} + G_\mathbf{n} & \ddots & \vdots \\
		\vdots & \ddots & \ddots & \mathrm{O}_m \\
		\gamma_{m-1} e^{-(m-1)\rho\Delta t} I_\mathbf{n} & \cdots & \gamma_1 e^{-\rho \Delta t} I_\mathbf{n} & \gamma_0 I_\mathbf{n} +  G_\mathbf{n}
	\end{array} \right]
	\in \C^{m\mathbf{n}\times m\mathbf{n}},
	\\
	\renewcommand{\arraystretch}{1.5}
	\tilde{\mathbf{u}} :=
	\left[ \begin{array}{c}
		\mathbf{u} \\
		\mathbf{u} \\
		\vdots \\
		\mathbf{u}
	\end{array} \right]
	\in\C^{m\mathbf{n}},
	\qquad
	\tilde{\mathbf{b}} :=
	\left[ \begin{array}{c}
		b_0 e^{-\rho\Delta t} \boldsymbol{\varphi} + \alpha_m q^{1} \mathbf{f} \\
		b_1 e^{-2\rho\Delta t} \boldsymbol{\varphi} + \alpha_m q^{2} \mathbf{f} \\
		\vdots \\
		b_{m-1} e^{-m\rho\Delta t} \boldsymbol{\varphi} + \alpha_m q^{m} \mathbf{f}
	\end{array} \right]
	\in\C^{m\mathbf{n}}.
\end{gather*}
The final time datum $\mathbf{u}^{N}$ is obtained by solving the forward problem. To model measurement noise, we introduce a perturbation of the form $\varepsilon \,\delta(x)$, where $\delta(x)$ is a random function uniformly distributed in $(-1,1)$. The resulting noisy datum is denoted by $\boldsymbol{\phi}$, representing a contaminated version of the exact final time observation.

Then, we apply the GMRES method to \eqref{eq:lin-syst} and report in Table \ref{table:gmres} the iteration counts and CPU times required by both the nonpreconditioned and preconditioned systems. We set the regularization parameter to $\lambda=5\cdot10^{-3}$, the noise level to $\varepsilon=0.01$, and consider different choices of the fractional orders $\xi$, $\eta_1$, and $\eta_2$. The initial guess is the zero vector, the tolerance is $\texttt{tol}=10^{-8}$, and the maximum number of iterations is equal to the size of the matrix.

Comparing the results obtained for the preconditioned and nonpreconditioned systems, we observe that the preconditioner $P_N$ consistently reduces the number of GMRES iterations for all tested values of the fractional orders. The reduction is particularly significant for $\xi=0.2$, $\eta_1=0.2$, and $\eta_2=0.8$, where the iteration counts of the nonpreconditioned system are the largest. For example, when $n_1=n_2=m=2^5$, the number of iterations decreases from $186$ to $39$, corresponding to a reduction of almost $80\%$. These results confirm the effectiveness of the proposed preconditioning strategy and are in agreement with the theoretical spectral analysis.

Furthermore, the preconditioner exhibits a remarkable robustness with respect to the problem dimensions. Indeed, for each fixed value of $n_1=n_2$, the number of iterations required by the preconditioned system remains nearly constant as $m$ increases. For instance, when $n_1=n_2=2^5$ and $(\xi,\eta_1,\eta_2)=(0.2,0.2,0.8)$, the preconditioned iteration counts vary only between $39$ and $40$, while the corresponding nonpreconditioned iteration counts remain close to $190$. Similar behaviour can be observed for all considered choices of the fractional orders.

In addition, for each fixed value of $m$, the number of iterations required by the preconditioned system grows moderately as the spatial discretization parameters $n_1$ and $n_2$ increase. In contrast, the nonpreconditioned iteration counts increase much more rapidly. This indicates that the preconditioner successfully mitigates the deterioration of the spectral properties associated with mesh refinement, leading to a substantially improved convergence behaviour.

The CPU timings generally reflect the reduction in iteration counts achieved by the preconditioner. Although the gain in computational time is less pronounced in some cases due to the additional cost associated with applying the preconditioner, the preconditioned framework remains competitive and considerably more scalable than the nonpreconditioned one as the problem dimensions increase.

Overall, the numerical experiments demonstrate that the proposed preconditioner provides a substantial acceleration of the GMRES method and maintains a robust performance across different discretization levels and fractional-order parameters.

\begin{table}[ht]
	\centering
	\caption{Iteration counts and CPU times required by the GMRES method for solving the linear system, with and without preconditioning, for different values of the fractional orders.}
	\label{table:gmres}
	\begin{tabular}{cccccccccccccc}
		\toprule
		& & \multicolumn{4}{c}{$\xi=0.2$} & \multicolumn{4}{c}{$\xi=0.5$} & \multicolumn{4}{c}{$\xi=0.8$}
		\\
		& & \multicolumn{4}{c}{$\eta_{1}=0.2,\;\;\eta_{2}=0.8$} & \multicolumn{4}{c}{$\eta_{1}=0.5,\;\;\eta_{2}=0.5$} & \multicolumn{4}{c}{$\eta_{1}=0.8,\;\;\eta_{2}=0.2$}
		\\
		\cmidrule(lr){3-6}\cmidrule(lr){7-10}\cmidrule(lr){11-14}
		& & \multicolumn{2}{c}{Iter} & \multicolumn{2}{c}{CPU} & \multicolumn{2}{c}{Iter} & \multicolumn{2}{c}{CPU} & \multicolumn{2}{c}{Iter} & \multicolumn{2}{c}{CPU}
		\\
		\cmidrule(lr){3-4}\cmidrule(lr){5-6}\cmidrule(lr){7-8}\cmidrule(lr){9-10}\cmidrule(lr){11-12}\cmidrule(lr){13-14}
		$n_{1}$ = $n_{2}$ & $m$ & - & $P_{\widetilde{\mathbf{N}}}$ & - & $P_{\widetilde{\mathbf{N}}}$ & - & $P_{\widetilde{\mathbf{N}}}$ & - & $P_{\widetilde{\mathbf{N}}}$ & - & $P_{\widetilde{\mathbf{N}}}$ & - & $P_{\widetilde{\mathbf{N}}}$
		\\
		\midrule
		$2^{2}$ & $2^{2}$ &32  & 8 & 0.02 &0.01  & 29 & 8 & 0.01 & 0.01 & 30 & 8 & 0.00 &0.00  \\
		$2^{2}$ & $2^{3}$ &34  & 8 &0.01  &0.00  & 33 & 8 & 0.01 & 0.01 &35  &8  & 0.00 &0.00  \\
		$2^{2}$ & $2^{4}$ & 34 & 8 & 0.01 & 0.00 &38  & 8 & 0.01 & 0.00 &47  & 8 &0.01  &0.01  \\
		$2^{2}$ & $2^{5}$ & 34 & 8 & 0.01 & 0.00 &42  & 8 & 0.01 & 0.01 &60 &8  & 0.01 &0.01  \\
		\midrule
		$2^{3}$ & $2^{2}$ & 59 & 17 & 0.01 &0.00  & 48 & 15 & 0.01 & 0.00 & 49 &16  &0.00  & 0.00 \\
		$2^{3}$ & $2^{3}$ & 57 & 16 & 0.01 & 0.01 & 50 & 16 & 0.01 & 0.01 &55  & 16 & 0.01 &0.00  \\
		$2^{3}$ & $2^{4}$ &  61& 16 & 0.01 & 0.01 & 54 & 16 & 0.01 & 0.01 &66 &15  &0.02  &0.01  \\
		$2^{3}$ & $2^{5}$ & 58 & 17 & 0.03 & 0.01 & 59 & 16 & 0.02 & 0.02 &  82& 16 &0.03  & 0.01 \\
		\midrule
		$2^{4}$ & $2^{2}$ & 98 & 27 & 0.04 & 0.02 &75  & 28 & 0.04 & 0.02 & 78 &27  &0.03  &0.01  \\
		$2^{4}$ & $2^{3}$ & 99 & 27 & 0.06 & 0.03 & 74 & 28 & 0.04 & 0.02 & 83 &27  &  0.04& 0.03 \\
		$2^{4}$ & $2^{4}$ & 97 & 28 & 0.12 & 0.04 & 79 & 28 & 0.09 & 0.04 &  92& 27 &0.11  &0.04  \\
		$2^{4}$ & $2^{5}$ & 98 & 27 & 0.22 & 0.08 & 84 & 28 & 0.16 & 0.07 &  113& 27 & 0.27 &0.08  \\
		\midrule
		$2^{5}$ & $2^{2}$ & 189 & 39 & 1.71 & 0.69 & 115 & 43 & 0.96 & 0.83 & 146 &38  & 1.18 &0.63  \\
		$2^{5}$ & $2^{3}$ & 189 & 40 & 2.00 & 0.82 & 114 & 42 & 1.02 & 0.83 &  138& 38 &1.58  & 1.63 \\
		$2^{5}$ & $2^{4}$ & 185 & 39 & 2.87 & 1.16 &115  & 41 & 1.40 & 1.13 & 144 &37  &2.04  &2.11  \\
		$2^{5}$ & $2^{5}$ &186  & 39 & 3.14 &1.10  & 116 & 40 & 2.13 & 1.15 & 149 &36  &3.10  &3.09  \\
		\bottomrule
	\end{tabular}
\end{table}

\subsection*{Study of the computational costs and further proposals}

In this subsection, we study in detail the computational cost of our preconditioned GMRES method and we discuss further potential proposals regarding the preconditioning.

When applying one step of the preconditioned GMRES with preconditioner \( P_{\widetilde{\mathbf{N}}} \), we have to solve a linear system with coefficient matrix \( P_{\widetilde{\mathbf{N}}} \) and we have to perform matrix-vector products with the original matrix \( A_{\widetilde{\mathbf{N}}} \). The related arithmetic costs are the following
\begin{itemize}
\item When performing one matrix-vector multiplication with matrix \( A_{\widetilde{\mathbf{N}}} \), we perform $\frac{m^2}{2} +O(m)$ products of a constant times a vector of size
$\nu(\mathbf{n})$, $\frac{m^2}{2} +O(m)$ addition between vectors of size $\nu(\mathbf{n})$, and $m+1$ products of the matrix $D_{\mathbf{n}}(a)B_{\mathbf{n}}$ times a vector: hence the corresponding arithmetic cost is equal to $m^2 \nu(\mathbf{n}) + C m \nu(\mathbf{n})\log [\nu(\mathbf{n})] + o(m^2 \nu(\mathbf{n}))=m^2 \nu(\mathbf{n}) + o(m^2 \nu(\mathbf{n}))$
arithmetic operations, if we make the natural assumption that $\log [\nu(\mathbf{n})] =o(m)$ and taking into account that $C$ is a constant;
\item When solving one linear system with coefficient matrix \( P_{\widetilde{\mathbf{N}}} \), we perform $\frac{m^2}{2} +O(m)$ products of a constant times a vector of size
$\nu(\mathbf{n})$, $\frac{m^2}{2} +O(m)$ addition between vectors of size $\nu(\mathbf{n})$, and $m+1$ solutions of linear systems with coefficient matrix $M_{\nu(\mathbf{n})}:=c I + D_{\mathbf{n}}(a)B_{\mathbf{n}}$ times a vector: hence the corresponding arithmetic cost is equal to $m^2 \nu(\mathbf{n}) + C' m \nu(\mathbf{n})\log [\nu(\mathbf{n})] + o(m^2 \nu(\mathbf{n}))=m^2 \nu(\mathbf{n}) + o(m^2 \nu(\mathbf{n}))$ arithmetic operations, if we make the natural assumption that $\log [\nu(\mathbf{n})] =o(m)$ and $C'$ depends linearly on the number of iterations we need preconditioning with $\widetilde P_{\nu(\mathbf{n})}:=c I + \hat a B_{\mathbf{n}}$ or with $\widetilde S_{\nu(\mathbf{n})}:=c I + \hat a S(B_{\mathbf{n}})$. As it can be seen from Table \ref{table:gmres2}, Table \ref{table:gmres3}, Table \ref{table:gmres4}, Table \ref{table:gmres5}, the number of iterations for both preconditioned GMRES is bounded by a constant independent of the matrix size $\nu(\mathbf{n})$ so that also $C'$ can be assumed to be constant.
\end{itemize}

We recall that $c I_{\mathbf{n}} + \hat a B_{\mathbf{n}}$ which is of Toeplitz lower triangular form and we have specialized techniques for lower triangular Toeplitz matrices of cost
$O(\nu(\mathbf{n})\log [\nu(\mathbf{n})])$ and this remains true for the circulant or $\omega$-circulant counterpart; see \cite{Bini,Diaz} and references therein. \ \\

\begin{table}[ht]
	\centering
	\caption{Iteration counts and CPU times of the GMRES method for solving the linear system with coefficient matrix \( \widetilde{M}_{\nu(\mathbf{n})} \) and $c=\gamma_0$, without preconditioner and with preconditioner \( \widetilde{P}_{\nu(\mathbf{n})} \), for different fractional orders and for \( a(x,y)=e^{x+y} \), with \( m=3 \).}
	\label{table:gmres2}
	\begin{tabular}{ccccccccccccc}
		\toprule
		&
		\multicolumn{4}{c}{$\xi=0.2$} &
		\multicolumn{4}{c}{$\xi=0.5$} &
		\multicolumn{4}{c}{$\xi=0.8$}
		\\
		&
		\multicolumn{4}{c}{$\eta_{1}=0.2,\;\eta_{2}=0.8$} &
		\multicolumn{4}{c}{$\eta_{1}=0.5,\;\eta_{2}=0.5$} &
		\multicolumn{4}{c}{$\eta_{1}=0.8,\;\eta_{2}=0.2$}
		\\
		\cmidrule(lr){2-5}\cmidrule(lr){6-9}\cmidrule(lr){10-13}
		
		&
		\multicolumn{2}{c}{Iter} &
		\multicolumn{2}{c}{CPU} &
		\multicolumn{2}{c}{Iter} &
		\multicolumn{2}{c}{CPU} &
		\multicolumn{2}{c}{Iter} &
		\multicolumn{2}{c}{CPU}
		\\
		\cmidrule(lr){2-3}\cmidrule(lr){4-5}
		\cmidrule(lr){6-7}\cmidrule(lr){8-9}
		\cmidrule(lr){10-11}\cmidrule(lr){12-13}
		
		$n_{1}=n_{2}$ &
		- & $\widetilde P_{\nu(\mathbf{n})}$ &
		- & $\widetilde P_{\nu(\mathbf{n})}$ &
		- & $\widetilde P_{\nu(\mathbf{n})}$ &
		- & $\widetilde P_{\nu(\mathbf{n})}$ &
		- & $\widetilde P_{\nu(\mathbf{n})}$ &
		- & $\widetilde P_{\nu(\mathbf{n})}$
		\\
		\midrule
		
		$2^{4}$ & 31 & 22 & 0.00 & 0.00 & 27 & 22 & 0.02 & 0.01 & 31 & 22 & 0.01 & 0.01 \\
		$2^{5}$ & 61 & 25 & 0.06 & 0.07 & 38 & 25 & 0.04 & 0.09 & 59 & 25 & 0.05 & 0.08 \\
		$2^{6}$ & 116 & 26 & 1.67 & 1.28 & 56 & 26 & 0.56 & 1.17 & 113 & 26 & 1.27 & 1.17 \\
		$2^{7}$ & 214 & 27 & 46.36 & 21.34 & 83 & 27 & 20.54 & 21.03 & 207 & 27 & 35.68 & 18.57 \\
		\bottomrule
	\end{tabular}
\end{table}
\begin{table}[ht]
	\centering
	\caption{Iteration counts and CPU times of the GMRES method for solving the linear system with coefficient matrix \( \widetilde{M}_{\nu(\mathbf{n})} \) and $c=\gamma_0$, without preconditioner and with preconditioner \( \widetilde{P}_{\nu(\mathbf{n})} \), for different fractional orders and for \( a(x,y)=1+{x+y} \), with \( m=3 \).}
	\label{table:gmres3}
	\begin{tabular}{ccccccccccccc}
		\toprule
		&
		\multicolumn{4}{c}{$\xi=0.2$} &
		\multicolumn{4}{c}{$\xi=0.5$} &
		\multicolumn{4}{c}{$\xi=0.8$}
		\\
		&
		\multicolumn{4}{c}{$\eta_{1}=0.2,\;\eta_{2}=0.8$} &
		\multicolumn{4}{c}{$\eta_{1}=0.5,\;\eta_{2}=0.5$} &
		\multicolumn{4}{c}{$\eta_{1}=0.8,\;\eta_{2}=0.2$}
		\\
		\cmidrule(lr){2-5}\cmidrule(lr){6-9}\cmidrule(lr){10-13}
		
		&
		\multicolumn{2}{c}{Iter} &
		\multicolumn{2}{c}{CPU} &
		\multicolumn{2}{c}{Iter} &
		\multicolumn{2}{c}{CPU} &
		\multicolumn{2}{c}{Iter} &
		\multicolumn{2}{c}{CPU}
		\\
		\cmidrule(lr){2-3}\cmidrule(lr){4-5}
		\cmidrule(lr){6-7}\cmidrule(lr){8-9}
		\cmidrule(lr){10-11}\cmidrule(lr){12-13}
		
		$n_{1}=n_{2}$ &
		- & $\widetilde P_{\nu(\mathbf{n})}$ &
		- & $\widetilde P_{\nu(\mathbf{n})}$ &
		- & $\widetilde P_{\nu(\mathbf{n})}$ &
		- & $\widetilde P_{\nu(\mathbf{n})}$ &
		- & $\widetilde P_{\nu(\mathbf{n})}$ &
		- & $\widetilde P_{\nu(\mathbf{n})}$
		\\
		\midrule
		
		$2^{4}$ & 29 & 15 & 0.00 & 0.00 & 24 & 14 & 0.00 & 0.00 & 28 & 14 & 0.00 & 0.00 \\
		$2^{5}$ & 50 & 15 & 0.02 & 0.02 & 32 & 15 & 0.01 & 0.03 & 49 & 15 & 0.02 & 0.02 \\
		$2^{6}$ & 88 & 16 & 0.59 & 0.41 & 44 & 16 & 0.30 & 0.41 & 86 & 15 & 0.59 & 0.37 \\
		$2^{7}$ & 156 & 19 & 14.98 & 7.02 & 61 & 16 & 6.87 & 7.18 & 153 & 16 & 18.62 & 7.41 \\
		\bottomrule
	\end{tabular}
\end{table}

\begin{table}[ht]
	\centering
	\caption{Iteration counts and CPU times of the GMRES method for solving the linear system with coefficient matrix \( \widetilde{M}_{\nu(\mathbf{n})} \) and $c=\gamma_0$, without preconditioner and with preconditioner \( \widetilde{S}_{\nu(\mathbf{n})} \), for different fractional orders and for \( a(x,y)=e^{x+y} \), with \( m=3 \).}
	\label{table:gmres4}
	\begin{tabular}{ccccccccccccc}
		\toprule
		&
		\multicolumn{4}{c}{$\xi=0.2$} &
		\multicolumn{4}{c}{$\xi=0.5$} &
		\multicolumn{4}{c}{$\xi=0.8$}
		\\
		&
		\multicolumn{4}{c}{$\eta_{1}=0.2,\;\eta_{2}=0.8$} &
		\multicolumn{4}{c}{$\eta_{1}=0.5,\;\eta_{2}=0.5$} &
		\multicolumn{4}{c}{$\eta_{1}=0.8,\;\eta_{2}=0.2$}
		\\
		\cmidrule(lr){2-5}\cmidrule(lr){6-9}\cmidrule(lr){10-13}
		
		&
		\multicolumn{2}{c}{Iter} &
		\multicolumn{2}{c}{CPU} &
		\multicolumn{2}{c}{Iter} &
		\multicolumn{2}{c}{CPU} &
		\multicolumn{2}{c}{Iter} &
		\multicolumn{2}{c}{CPU}
		\\
		\cmidrule(lr){2-3}\cmidrule(lr){4-5}
		\cmidrule(lr){6-7}\cmidrule(lr){8-9}
		\cmidrule(lr){10-11}\cmidrule(lr){12-13}
		
		$n_{1}=n_{2}$ &
		- & $\widetilde S_{\nu(\mathbf{n})}$ &
		- & $\widetilde S_{\nu(\mathbf{n})}$ &
		- & $\widetilde S_{\nu(\mathbf{n})}$ &
		- & $\widetilde S_{\nu(\mathbf{n})}$ &
		- & $\widetilde S_{\nu(\mathbf{n})}$ &
		- & $\widetilde S_{\nu(\mathbf{n})}$
		\\
		\midrule
		
		$2^{4}$ & 31 & 25 & 0.00 & 0.01 & 27 & 21 & 0.02 & 0.02 & 31 & 24 & 0.01 & 0.01 \\
		$2^{5}$ & 61 & 27 & 0.06 & 0.03 & 38 & 24 & 0.04 & 0.03 & 59 & 26 & 0.05 & 0.03 \\
		$2^{6}$ & 116 & 29 & 1.67 & 0.42 & 56 & 25 & 0.56 & 0.37 & 113 & 27 & 1.27 & 0.32 \\
		$2^{7}$ & 214 & 29 & 46.36 & 6.64 & 83 & 25 & 20.56 & 4.78 & 207 & 27 & 35.68 & 4.69 \\
		\bottomrule
	\end{tabular}
\end{table}

\begin{table}[ht]
	\centering
	\caption{Iteration counts and CPU times of the GMRES method for solving the linear system with coefficient matrix \( \widetilde{M}_{\nu(\mathbf{n})} \) and $c=\gamma_0$, without preconditioner and with preconditioner \( \widetilde{S}_{\nu(\mathbf{n})} \), for different fractional orders and for \( a(x,y)=1+{x+y} \), with \( m=3 \).}
	\label{table:gmres5}
	\begin{tabular}{ccccccccccccc}
		\toprule
		&
		\multicolumn{4}{c}{$\xi=0.2$} &
		\multicolumn{4}{c}{$\xi=0.5$} &
		\multicolumn{4}{c}{$\xi=0.8$}
		\\
		&
		\multicolumn{4}{c}{$\eta_{1}=0.2,\;\eta_{2}=0.8$} &
		\multicolumn{4}{c}{$\eta_{1}=0.5,\;\eta_{2}=0.5$} &
		\multicolumn{4}{c}{$\eta_{1}=0.8,\;\eta_{2}=0.2$}
		\\
		\cmidrule(lr){2-5}\cmidrule(lr){6-9}\cmidrule(lr){10-13}
		
		&
		\multicolumn{2}{c}{Iter} &
		\multicolumn{2}{c}{CPU} &
		\multicolumn{2}{c}{Iter} &
		\multicolumn{2}{c}{CPU} &
		\multicolumn{2}{c}{Iter} &
		\multicolumn{2}{c}{CPU}
		\\
		\cmidrule(lr){2-3}\cmidrule(lr){4-5}
		\cmidrule(lr){6-7}\cmidrule(lr){8-9}
		\cmidrule(lr){10-11}\cmidrule(lr){12-13}
		
		$n_{1}=n_{2}$ &
		- & $\widetilde S_{\nu(\mathbf{n})}$ &
		- & $\widetilde S_{\nu(\mathbf{n})}$ &
		- & $\widetilde S_{\nu(\mathbf{n})}$ &
		- & $\widetilde S_{\nu(\mathbf{n})}$ &
		- & $\widetilde S_{\nu(\mathbf{n})}$ &
		- & $\widetilde S_{\nu(\mathbf{n})}$
		\\
		\midrule
		
		$2^{4}$ & 29 & 17 & 0.00 & 0.00 & 24 & 13 & 0.00 & 0.00 & 28 & 16 & 0.00 & 0.00 \\
		$2^{5}$ & 50 & 18 & 0.02 & 0.01 & 32 & 14 & 0.01 & 0.01 & 49 & 17 & 0.02 & 0.01 \\
		$2^{6}$ & 88 & 18 & 0.59 & 0.14 & 44 & 14 & 0.30 & 0.10 & 86 & 17 & 0.59 & 0.12 \\
		$2^{7}$ & 156 & 18 & 14.98 & 1.73 & 61 & 15 & 6.87 & 1.68 & 153 & 17 & 18.62 & 2.49 \\
		\bottomrule
	\end{tabular}
\end{table}
\FloatBarrier

Finally, we introduce an alternative Strang-like preconditioner \( S_{\widetilde{\mathbf{N}}} \) for the linear system \eqref{eq:lin-syst}.
\smallskip

\noindent A Strang-like preconditioner for our block structure is defined by replacing each lower triangular Toeplitz matrix in the various blocks with its natural circulant approximation and by replacing coefficient $a(x,y)$ with its integral average $\hat a$ so that $D_{\mathbf{n}}(a)$ is replaced by
$\hat aI_{\mathbf{n}}$ and $G_{\mathbf{n}}$ reduces to $$G_{\mathbf{n}} = \hat a B_{\mathbf{n}}, \qquad B_{\mathbf{n}}=\left(B_{\mathbf{n}}^{x}+B_{\mathbf{n}}^{y}\right),$$
that is
\begin{equation}\label{eq:strang-precond}
	\renewcommand{\arraystretch}{1.5}
	S_{\widetilde{\mathbf{N}}} := \left[ \begin{array}{c|c}
		S(U_m) \otimes I_{\mathbf{n}} + I_m\otimes a S(B_{\mathbf{n}}) & \mathrm{O}_{\mathbf{n}m \times \mathbf{n}} \\
		\hline
		\mathbf{e}_m^T \otimes I_\mathbf{n} & \frac{\lambda}{\alpha_{m}} D_{\mathbf{n}}(a) S(B_{\mathbf{n}})
	\end{array} \right].
\end{equation}
Since circulant matrices are diagonalizable by the Fourier matrices, we write
\begin{equation*}
	S(U_m) = F_m \Lambda_m^{(1)} F_m^H,
	\qquad
	S(B_{\mathbf{n}}) = F_{\mathbf{n}} \Lambda_{\mathbf{n}}^{(2)} F_{\mathbf{n}}^H,
\end{equation*}
where $\Lambda_m^{(1)}$ and $\Lambda_\mathbf{n}^{(2)}$ are the diagonal matrices containing the eigenvalues and can be computed starting from the coefficients of $U_m$, $B_{\mathbf{n}}$ via the Fast Fourier Transform (FFT).
\smallskip

Define the unitary transformation

\begin{equation*}
	\renewcommand{\arraystretch}{1.5}
Q := \left[ \begin{array}{c|c}
		F_m\otimes F_{\mathbf{n}} & \mathrm{O}_{{\mathbf{n}}m \times {\mathbf{n}}} \\
		\hline
		\mathrm{O}_{{\mathbf{n}}\times{\mathbf{n}}m} &  F_{\mathbf{n}}
	\end{array} \right],
\end{equation*}
Then
\begin{equation}\label{eq:p11p}
	\renewcommand{\arraystretch}{1.5}
	Q^{H} S_{\widetilde{\mathbf{N}}} Q := \left[ \begin{array}{c|c}
		\Lambda_m^{(1)}\otimes I_{\mathbf{n}} + I_m\otimes  a \Lambda_\mathbf{n}^{(2)}  & \mathrm{O}_{{\mathbf{n}}m \times {\mathbf{n}}} \\
		\hline
		\mathbf{e}_m^T F_m^H \otimes I_\mathbf{n} & \frac{\lambda}{\alpha_m} a \Lambda_{\mathbf{n}}^{(2)}
	\end{array} \right].
\end{equation}

The preconditioner $S_{\widetilde{\mathbf{N}}}$ is cheaper that that $P_{\widetilde{\mathbf{N}}}$, but it cannot ensure the clustering at $1$ due to GLT arguments since
\[
\{S_{\widetilde{\mathbf{N}}}^{-1}A_{\widetilde{\mathbf{N}}}\}_{\widetilde{\mathbf{N}}} \sim_{\glt,\sigma} \frac{h_\xi(\theta_1) + a(x,y) \left(\nu_{1}g_{\eta_{1}}(\theta_{2})+\nu_{2}g_{\eta_{2}}(\theta_{3})\right)}{h_\xi(\theta_1) + \hat a \left(\nu_{1}g_{\eta_{1}}(\theta_{2})+\nu_{2}g_{\eta_{2}}(\theta_{3})\right)}
\]
and hence it is less effective that $P_{\widetilde{\mathbf{N}}}$ in terms of iteration count, also becausethe asymptotic behavior of the eigenvalues is not favourable. A more refined analysis is left for future investigations as indicated in Section \ref{sec:end}.

\section{Conclusions} \label{sec:end}
We investigated the inverse source problem for a two-dimensional space-time FDE with variable coefficients. The ill-posedness was effectively handled through the quasi-boundary value method, ensuring stable reconstruction in the presence of noisy final-time data. A finite difference discretization led to a large-scale all-at-once system with multilevel Toeplitz-like structure. Using GLT theory, we derived the symbol of the matrix sequence and characterized its spectral distribution. This analysis provided the foundation for constructing efficient preconditioners. We introduced a block triangular preconditioner and proved that the corresponding preconditioned matrix sequences are spectrally distributed as one in the GLT, singular value and eigenvalue sense. As expected for multilevel problems \cite{MultiNo}, the clustering is not strong and hence the GMRES iteration count mildly increases with the problem size. Numerical results confirmed the theory: in fact the proposed preconditing reduced iteration counts and computational cost, while maintaining robustness across different fractional orders. The convergence behavior depends mainly on the spatial discretization, with limited sensitivity to temporal refinement. The proposed framework combines regularization, structure-preserving discretization, and GLT-based spectral analysis. It provides an efficient and scalable approach for inverse problems in fractional diffusion models.

Future work will address extensions to more complex geometries, finite elements of high order or isogeometric analysis with intermediate regulaity, higher dimensions, and adaptive parameter selection, where a program for the first two items has been sketched in Remark \ref{remark:future directions}, while taking inspiration from the variable-coeffient approach in \cite{Serra1999} and from the $\omega$-circulant approach \cite{Diaz}, other options for the choice of the preconditioners should be considered starting from the brief discussion at the end of Section \ref{sec:num}.

\section*{Acknowledgments}

The research of Stefano Serra-Capizzano is supported by the Italian National Agency INdAM-GNCS, by the PRIN-PNRR project \lq\lq MATH-ematical tools for predictive maintenance and PROtection of CULTtural heritage (MATHPROCULT)'' (code P20228HZWR, CUP J53D23003780006), and by the INdAM-GNCS Project \lq\lq Metodi strutturati per il signal processing avanzato'' (CUP E53C25002010001). Furthermore Stefano Serra-Capizzano is grateful for the support of the Laboratory of Theory, Economics and Systems – Department of Computer Science at Athens University of Economics and Business.

\end{document}